\documentclass[reqno, a4paper]{amsart}
\usepackage{amsmath, amssymb,amsthm,amsfonts,dsfont}

\usepackage[T1]{fontenc}
\usepackage{enumerate}
\usepackage{color}
\usepackage{caption}
\usepackage{accents}
\usepackage{tikz}
\usetikzlibrary{patterns}
\usepackage{array}

\newtheorem{theorem}{Theorem}

\newtheorem{definition}[theorem]{Definition}
\newtheorem{lemma}[theorem]{Lemma}
\newtheorem{proposition}[theorem]{Proposition}

\newtheorem*{conjecture*}{Conjecture}

\theoremstyle{remark}
\newtheorem{remark}[theorem]{Remark}
\newtheorem{example}[theorem]{Example}

 \newcommand{\eps}{\varepsilon}

 \renewcommand{\phi}{\varphi}

\newcommand{\E}{\mathbb{E}}
\renewcommand{\P}{\mathbb{P}}

\newcommand{\C}{\mathcal{C}}

\newcommand{\R}{\mathbb{R}}

\newcommand{\cl}{\textsc{cl}}
\newcommand{\op}{\textsc{op}}
\newcommand{\G}{\mathcal G}

\newcommand{\cE}{\mathcal{E}}

\renewcommand{\eps}{\varepsilon}
\renewcommand{\epsilon}{\varepsilon}

\newcommand{\fourIdx}[5]{%
\setbox1=\hbox{\ensuremath{^{#1}}}%
 \setbox2=\hbox{\ensuremath{_{#2}}}%
 \setbox5=\hbox{\ensuremath{#5}}%
 \hspace{\ifnum\wd1>\wd2\wd1\else\wd2\fi}%
 \ensuremath{\copy5^{\hspace{-\wd1}\hspace{-\wd5}#1\hspace{\wd5}#3}%
 _{\hspace{-\wd2}\hspace{-\wd5}#2\hspace{\wd5}#4}%
 }}

\numberwithin{equation}{section}
\numberwithin{theorem}{section}

\newcommand{\SG}{\mathsf{SG}}

\renewcommand{\subset}{\subseteq}

\newcommand{\SSG}{{\mathsf{SG}_2}}
\newcommand{\mg}{\mathsf{MCpl}}

\usepackage{subcaption}
\usepackage{graphicx}

\renewcommand{\llcorner}{\upharpoonright}

\begin{document}

\title[Transformations of MOT and SEP]{Monotonicity Preserving Transformations of MOT and SEP}

\author[M. Huesmann]{Martin Huesmann} 
\email{huesmann@iam.uni-bonn.de}

\author[F. Stebegg]{Florian Stebegg}
\email{florian.stebegg@columbia.edu}
\thanks{We thank Mathias Beiglb\"ock and Alex Cox for numerous valuable discussions. MH gratefully acknowledges support by the CRC1060 and the Hausdorff Center for Mathematics.}

\date{\today}
\begin{abstract}
Recently, \cite{BeJu16, BeNuTo16} established that optimizers to the martingale optimal transport problem (MOT) are concentrated on $c$-monotone sets. In this article we characterize monotonicity preserving transformations revealing certain symmetries between optimizers of MOT for different cost functions. Due to the intimate connection of MOT and the Skorokhod embedding problem (SEP) these transformations are also monotonicity preserving and disclose symmetries for certain solutions to the optimal SEP. 
Furthermore, the SEP picture allows to easily understand the geometry of these transformations once we have established the SEP counterparts to the known solutions of MOT based on the monotonicity principle for SEP which in turn allows to directly read off the structure of the MOT optimizers.

\medskip

\end{abstract}

\maketitle
\section{Introduction}

Given probabilities $\mu$ and $\nu$ on $\R$ and a \emph{cost function} $c:\R\times\R\to\R$ the martingale optimal transport problem (MOT) is to minimize
$$\int c(x,y)\ dQ(x,y)$$
among all \emph{martingale couplings} $Q$ of $\mu$ and $\nu$, i.e.\ among all couplings satisfying $\int y\ dQ_x(y)=x$ for $\mu$-a.e.\ $x$, where $(Q_x)_x$ denotes a disintegration of $Q$ with respect to $\mu$. The articles \cite{BeJu16, BeNuTo16} showed that similar to the classical optimal transport problem optimizers are characterized by a local optimality condition in that every optimizer is concentrated on a $c$-monotone set $\Xi$. If $c$ is sufficiently nice, also the reverse holds. Every martingale coupling concentrated on $\Xi$ is optimal between its marginals \cite{BeJu16, Gr16}. This indicates that in order to understand the optimizers of MOT it is necessary to understand the geometry of $c$-monotone sets.

One way of understanding this geometry is to understand the rigidity of its characterizing properties under transformations. On the other hand, transformations of martingales play a prominent role in modern stochastic. The aim of this article is to understand the rigidity of c-monotone sets as well as the stability of the martingale property under so called \emph{monotonicity preserving transformations.} As a byproduct these transformations reveal various new symmetries between solutions to MOT for different cost functions.

%The aim of this article is to characterize monotonicity preserving transformations revealing symmetries between solutions to MOT for different cost functions.
More precisely, given intervals $I,J,I',J'$ each bijective map $T:I\times J\to I'\times J'$ together with a positive function $h$ induces a transformation $\tau$ of measures on $I\times J$ to measures on $I'\times J'$ by setting
$$\int g(x',y') d\tau(\pi)(x',y'):=\int g\circ T(x,y) h(x,y) d\pi(x,y)$$
and a transformation of functions $c$ on $I\times J$ to functions $c'$ on $I'\times J'$ by setting
$$c'(x',y'):=(c/h)\circ T^{-1}(x',y')\ .$$

The purpose of $T$ is to move mass from a point $(x,y)$ to a point $(x',y')$, whereas $h$ rescales the mass. Because $h$ is
taken to be strictly positive, the structure of the 'support' of a measure $\pi$ is only transformed via $T$.

Roughly speaking, the pair $(T,h)$ is a monotonicity preserving transformation if it maps  a $c$-monotone set $\Xi$ into a $c'$-monotone set $\Xi'$ (for a precise definition we refer to Section \ref{sec:trafo}).
It turns out that there are only very few such transformations:

\begin{theorem}\label{thm:montrafo}
 Let $T(x,y) = (s(x,y),t(x,y)):I\times J\to I'\times J'$ be bijective and $h: I \times J \to (0,\infty)$. Assume that $(T,h)$ is a monotonicity preserving transformation.
Then $s$ does not depend on $y$ and $t$ and $h$ do not depend on $x$.
\begin{enumerate}[(1)]
\item either $h$ is constant and $t$ is affine in $y$;
\item or $h$ is affine in $y$ ($h(y) = c(y - b)$) and $t$ is of the form $t(y) = a/(y-b)$.
\end{enumerate}
\end{theorem}

We stress that a monotonicity preserving transformation does not need to (and in general will not) preserve the marginals or the martingale property. However, if we additionally require the transformation to preserve the martingale property, we get a new characterization of a well known transformation, the change of numeraire transformation which has been studied for MOT in \cite{CaLaMa16}.

\begin{theorem}\label{thm:change of numeraire}
Let $T:I\times J\to I'\times J'$ be  bijective and $h: I \times J \to (0,\infty)$. Assume that $(T,h)$ is a monotonicity preserving transformation. Then $T$ preserves martingales if and only if it is of the following form,
\begin{enumerate}[(1)]
\item either $T(x,y)=(ax+b,ay+b), h\equiv 1$;
\item or  $T(x,y)=(a/(x-b),a/(y-b))$ and $h(y)=c(y-b)$.
\end{enumerate}
The second transformation is only well-defined for $I,J \subseteq (b,\infty)$
and for a probability measure $\pi$, we only have that $\tau(\pi)$ is a probability measure if $c^{-1} + b = \int x \,d\pi_1 = \int y \,d\pi_2$.
\end{theorem}
Further examples of monotonicity preserving transformations are the mirror transformations in \cite[Remark 5.2]{HeTo13}. We emphasize that---using the transformation $T(x,y)=(-x,y), h\equiv 1$---our results also reveal certain symmetries between optimizers of $|x-y|$ and those of $|x+y|$.

We note that the assumption on the domain of the map $T$ in Theorem \ref{thm:montrafo} is not very restrictive since this is precisely the shape of the domain of an irreducible component of the pair $(\mu,\nu)$ as shown in \cite[Theorem 8.4]{BeJu16}.
\medskip\\
Already in one of the first articles on MOT \cite{HoNe12} the intimate connection of MOT and the Skorokhod embedding problem (SEP) was used to construct the optimizer for the cost function $c(x,y)=-|x-y|$.  The relation between MOT and SEP was used in \cite{BeHeTo15}
to give a short proof of the uniqueness of left monotone martingale couplings.
The reason for this fruitful connection is simple: Any discrete time martingale can be interpolated to a continuous time martingale 
\cite{He77} and, hence, by the Dambis-Dubins-Schwarz theorem the optimal SEP can be seen as a continuous time version of MOT.

Recently, \cite{BeCoHu16} established that solutions of the \emph{optimal SEP}, i.e.\ solutions of SEP optimizing a given cost functional $\gamma$, are concentrated on $\gamma$-monotone sets $\Gamma$. Understanding these sets for different $\gamma$ allows to construct various solutions to the optimal  SEP including all known solutions as special cases. Viewing the optimal SEP as a continuous time version of MOT the transformations characterized in Theorem \ref{thm:montrafo} induce transformations of the $\gamma$-monotone sets for the optimal SEP for the specific class of functionals $\gamma$ of the form $\gamma(g,t)=\gamma(g(0),g(t))$ for $g\in C[0,t]$ and, thus, Theorem \ref{thm:montrafo} discloses symmetries between different solutions of SEP.

Furthermore, the understanding of the geometry of the $\gamma$-monotone sets for the SEP counterparts to the MOT problems allows us to give a simple derivation of the known solutions to MOT. From the specific structure of the SEP solutions one can directly read off the defining properties of the optimal martingale couplings. Building on this, it is straightforward to deduce the optimizers for the transformed MOT problem, i.e.\ for the MOT with respect to the cost function $c'=(c/h)\circ T^{-1}$. 

It is desirable to have a result similar to Theorem \ref{thm:montrafo} also for the general optimal SEP. However, it seems that a necessary ingredient is still missing, a 'full monotonicity principle'. 
We give a conjecture on this principle and observe that the conjecture holds for the special case of cost functions $\gamma$ considered in this article due to the connection of SEP and MOT.

\subsection{Related literature} The MOT was introduced in  \cite{BeHePe12} where also a duality result was established for lower semicontinuous cost functions. The duality result was extended in  \cite{BeNuTo16} to Borel measurable cost functions building on a deep understanding of MOT developed in \cite{BeJu16} which develops a theory parallel to classical optimal transport. Additionally, \cite{BeJu16} construct the optimizer for cost functions of the form $c(x,y)=h(y-x)$ with $h'''>0$ ($h'''<0$ resp.), the so called left (right resp.) monotone couplings. It was shown in \cite{HeTo13} that the left (right resp.) monotone couplings are also optimizers for the cost functions $c$ satisfying the generalized Spence--Mirrlees condition $c_{xyy}<0$ ($c_{xyy}>0$ resp.). Previously, \cite{HoNe12, HoKl15} constructed the optimizers for the cost functions $c(x,y)=\pm|x-y|$ not satisfying the Spence--Mirrlees condition.

In higher dimension, first  results have been established by \cite{GhKiLi15, Li16}. The duality result was extended in various directions to continuous time, e.g. \cite{GaHeTo14, DoSo12, DoSo15, HoOb15, BeCoHuPePr15, GuTaTo16} and references therein.

\subsection{Outline}
In Section \ref{sec:prel} we recall some facts on martingale optimal transport as well as the monotonicity principle of \cite{BeCoHu16} that will allow us to give a simple geometric explanation of the structure of all known solutions to MOT in Section \ref{sec:mot}. In Section \ref{sec:trafo} we prove Theorems \ref{thm:montrafo} and \ref{thm:change of numeraire} and show that their geometric consequences can be easily unraveled using the SEP picture of martingale optimal transport. We end by stating a conjecture on a full monotonicity principle for Skorokhod embedding in Section \ref{sec:conj}.

\section{Preliminaries}\label{sec:prel}

\subsection{Martingale Optimal Transport}\label{ssec:mot}
Let $\mu,\nu\in\mathcal P^1(\R)$ (where $\mathcal P^1(X)$ denotes the set of all probability measures on a space $X$ with finite first moments). Let $\mathsf{Cpl}(\mu,\nu)$ be the set of all couplings of $\mu$ and $\nu$, i.e.\ measures on $\R^2$ with marginals $\mu$ and $\nu$, and let $\mg(\mu,\nu)\subset\mathsf{Cpl}(\mu,\nu)$ be the subset of all martingale couplings, i.e.\ all couplings under which the coordinate process $(X,Y)$ becomes a martingale, where $X(x,y)=x$ and $Y(x,y)=y$. By Strassen's Theorem $\mg(\mu,\nu)\neq\emptyset$ iff $\mu$ and $\nu$ are increasing in convex order, i.e.\ $\int \phi\ d\mu\leq\int \phi\ d\nu$ for all convex functions $\phi$. A measure $Q\in \mathsf{Cpl}(\mu,\nu)$ is a martingale coupling iff for any disintegration $(Q_x)_x$ of $Q$ with respect to $\mu$ it holds that $\mu$-a.s.
$$\int y \ Q_x(dy)=x\ .$$
Let $c:\R\times \R\to\R$ be some \textit{cost function}. Then, the martingale optimal transport problem is to solve
\begin{equation}\label{eq:MOT}
P^{mg}_{\mu,\nu}(c)=\inf_{Q\in\mg(\mu,\nu)} \int c(x,y)\ Q(dx,dy)\ .\tag{\sf{MOT}}
\end{equation}
It is not hard to see  that $\mg(\mu,\nu)$ is compact with respect to to the weak topology and, hence, \eqref{eq:MOT} admits a minimizer if $c$ is lower semicontinuous and bounded from below since then the functional
$$ Q\mapsto \int c(x,y)\ Q(dx,dy)$$
is lower semicontinuous by the Portmanteau theorem.  Beiglb\"ock, Nutz and Touzi showed in \cite{BeNuTo16} the following geometric characterization of optimizers to \eqref{eq:MOT}.
\begin{theorem}[{\cite[Corollary 7.8]{BeNuTo16}}]\label{thm:fullcm}
 Assume $P^{mg}_{\mu,\nu}(c)<\infty$. Then there exists a Borel set $\Xi\subset \R^2$ with the following properties.
\begin{enumerate}
\item A measure $Q\in\mg(\mu,\nu)$ is concentrated on $\Xi$ iff it is optimal for $P^{mg}_{\mu,\nu}(c)$.
\item Let $\bar\mu,\bar\nu$ be probabilities increasing in convex order. If $Q\in\mg(\bar\mu,\bar\nu)$ is concentrated on $\Xi$, then $Q$ is optimal for $P^{mg}_{\bar\mu,\bar\nu}(c).$
\end{enumerate}
\end{theorem}
In this article we will use a local version which was derived previously in \cite{BeJu16}, see also \cite{Gr16} for a small extension. To state it we need to introduce the notion of \emph{competitor} which is also central in this paper.
\begin{definition}\label{def:monotone}
 \begin{enumerate}
  \item Let $\alpha$ and $\beta$ be two finite measures on $\R^2$. Then $\beta$ is called \emph{competitor} of $\alpha$ if it has the same marginals $\alpha_0$ and $\alpha_1$ and for any disintegrations $(\alpha_x)_x$ and $(\beta_x)_x$ of $\alpha$ resp.\ $\beta$ with respect to the first marginal $\alpha_0$ it holds for $\alpha_0$ a.e. $x$  that 
\begin{equation}\label{eq:compdef}
 \int y\ \alpha_x(dy) = \int y\ \beta_x(dy)\ .
\end{equation}
\item A set $\Xi\subset \R^2$ is called $c$-monotone (or just monotone) if for any finite measure $\alpha$ concentrated on $\Xi$ with $|\mbox{supp}(\alpha)|<\infty$ and any competitor $\beta$ of $\alpha$ it holds that 
$$ \int c\ d\alpha \leq \int c\ d\beta\ .$$
 \end{enumerate}
\end{definition}
Then we have the following result:
\begin{theorem}[{\cite[Lemma 1.11]{BeJu16}, \cite[Theorem 1.3]{Gr16}}]\label{thm:varlemma}
 Let $Q$ be an optimizer of \eqref{eq:MOT} with $P^{mg}_{\mu,\nu}(c)<\infty$. Then there is a $c$-monotone set $\Xi$ satisfying $Q(\Xi)=1.$ Moreover, if $c$ is upper-semicontinuous and bounded from above by integrable functions (with respect to $\mu$ resp.\ $\nu$) this condition is also sufficient.
\end{theorem}

\subsection{The optimal Skorokhod embedding problem}\label{sec:OptSEP}

Fix $\mu,\nu\in\mathcal P(\R)$ increasing in convex order. For notational convenience we assume that $\nu$ has second moment. For the general case we refer to \cite{BeCoHu16}. Let $(\Omega,\G,(\G_t)_{t\geq 0},\P)$ be some complete stochastic base which is sufficiently rich to support a Brownian motion $B$ starting in $\mu$  and a uniform $\G_0$-measurable random variable independent of $B$. Let $S:=\{(f,s):f\in C([0,s])\}$  and fix $\gamma:S\to\R$. We put 
$\gamma_t:=\gamma((B_s)_{s\leq t},t)$ and note that this is an optional process. For $(f,s),(g,t)\in S$ we define the concatenation of $f$ and $g$ by
$$ f\oplus g(u)=\begin{cases}
                 f(u) & \text{ if } u\leq s,\\
f(s)+g(u-t)-g(t) & \text{ if } u > s .
                \end{cases}$$
For $(f,s)\in S$ we set $\gamma^{(f,s)\oplus}_t:=\gamma(f\oplus (B_u)_{u\leq t},s+t).$ 

The optimal Skorokhod embedding problem is to find a minimizer of 
\begin{equation}\label{eq:OptSEP}
 \tau\mapsto \E[\gamma_\tau] \tag{\sf{OptSEP}}
\end{equation}
 among all $\G$ stopping times $\tau$ such that $B_\tau\sim \nu$ and $(B_{t\wedge \tau})_{t\geq 0}$ is uniformly integrable. We write $\mathsf{Opt}_\gamma$ for the set of all optimizer of $\eqref{eq:OptSEP}$. Considering another functional $\tilde \gamma:S\to\R$ we say that $\hat\tau\in\mathsf{Opt}_\gamma$ is a secondary optimizer if it minimizes
\begin{equation}\label{eq:OptSEP2}
 \tau\mapsto \E[\tilde\gamma_\tau] \tag{$\mathsf{OptSEP}_2$}
\end{equation}
among all $\tau\in\mathsf{Opt}_\gamma$. We will say that $\eqref{eq:OptSEP}$ is well posed if $\E[\gamma_\tau]$ exists with values in $(-\infty,\infty]$ for all $\tau$ satisfying $B_\tau\sim\nu$ and $(B_{t\wedge \tau})_{t\geq 0}$ is uniformly integrable and it is finite for one such $\tau$; similarly for $\eqref{eq:OptSEP2}$. It is not hard to see that an optimizer to $\eqref{eq:OptSEP2}$ exists if for example $\gamma,\tilde\gamma$ are lower semicontinuous and bounded from below (see \cite[Theorem 4.1]{BeCoHu16}).

\begin{definition}\label{def:SG}
  \label{def:SG2B} The pair $\big((f,s), (g,t)\big)\in S\times S$ constitutes a \emph{stop-go pair}, written $\big((f,s), (g,t)\big)\in\SG$, iff $f(s)=g(t)$, and 
 for \emph{every} $(\G_t)_{t \ge 0}$-stopping time $\sigma$ which satisfies $0 < \E[\sigma] < \infty$, 
  \begin{align} \label{BPIneq1}
    \E\big[\big(\gamma^{(f,s)\oplus }\big)_{\sigma}\big]\ +\ \gamma(g,t)\quad > \quad \gamma(f,s) \ +\ \E\big[\big(\gamma^{(g,t)\oplus }\big)_{\sigma}\big],
  \end{align}
  whenever both sides are well defined, and the left-hand side  is finite. We say that $\big((f,s), (g,t)\big)\in S\times S$ constitutes a \emph{secondary stop-go pair}, written $\big((f,s), (g,t)\big)\in\SG_2$, iff $f(s)=g(t)$, and  for \emph{every} $(\G_t)_{t \ge 0}$-stopping time $\sigma$ which satisfies $0 < \E[\sigma] < \infty$ the inequality \eqref{BPIneq1} holds with $\geq$   and if
  \begin{align}\label{BPIneq2}
    \E\big[\big(\gamma^{(f,s)\oplus }\big)_{\sigma}\big]\ +\ \gamma(g,t)\quad = \quad \gamma(f,s) \ +\ \E\big[\big(\gamma^{(g,t)\oplus }\big)_{\sigma}\big]
  \end{align}
  then
  \begin{align}
    \E\big[\big(\tilde\gamma^{(f,s)\oplus }\big)_{\sigma}\big]\ +\ \tilde\gamma(g,t)\quad&>\quad \tilde\gamma(f,s) \ +\ \E\big[\big(\tilde\gamma^{(g,t)\oplus }\big)_{\sigma}\big], \label{BPIneq3}
  \end{align}
  whenever both sides are well-defined and the left-hand side (of \eqref{BPIneq3}) is finite. 
\end{definition}

\begin{definition} \label{def:ggtildeMon}
We say that $\Gamma\subset S$ is $\gamma$-monotone  if 
\begin{align}\label{GM3}\SG \cap \big( \Gamma^< \times \Gamma\big)=\emptyset,
\end{align} 
where $\Gamma^<:=\{ (f,s) \in S: \exists (g,t)\in\Gamma, t>s, g_{\llcorner [0,s]}\equiv f\}.$ We say that $\Gamma\subset S$ is $ \tilde\gamma| \gamma$-monotone if \eqref{GM3} holds with $\SSG$ in place of $\SG$.
\end{definition}

The following theorem is proven in \cite[Theorems 5.7 and 5.16]{BeCoHu16}.

\begin{theorem}[Monotonicity Principle]\label{GlobalLocal3}
  Let $\gamma, \tilde \gamma:S\to\R$ be Borel measurable, suppose that \eqref{eq:OptSEP} is well posed and that $\hat \tau$ is an optimizer.
  Then there exists a \emph{$\gamma$-monotone} Borel set $\Gamma\subset S$ such that $\P$-a.s.
\begin{align}\label{GammaSupport2} ((B_t)_{t\leq\hat\tau},\hat\tau)\in\Gamma\;.\end{align}
If \eqref{eq:OptSEP2} is well posed and $\hat\tau$ is also an optimizer to this problem, then there exists a
\emph{$\tilde\gamma| \gamma$-monotone} Borel set $\Gamma\subset S$ such that $\P$-a.s.
\begin{align}\label{GammaSupport3} ((B_t)_{t\leq\hat\tau},\hat\tau)\in\Gamma\;.\end{align}
\end{theorem}

\section{MOT via SEP}\label{sec:mot}

To show the connection between MOT and the SEP, we sketch how the properties of solutions to MOT for well-known cost functions can
be derived using the monotonicity results for SEP. We will do this in detail for
the case of Spence--Mirrlees cost functions. The proof for the form of maximizing and minimizing transports
for $|x-y|$ is very similar and will therefore be reduced to a sketch.

\begin{proposition}\label{prop:SpMi}
Assume that $\mu(\{x\})=0$ for all $x\in\R$. Then, there exists a unique stopping time $\tau_{BJ}$ which minimizes
\[\E[c(B_0,B_{\tau})]\]
over all solutions to SEP for any cost function $c$ satisfying the Spence--Mirrlees condition $c_{xyy}<0$. It is of the form 
$\tau_{BJ} = \inf\{t > 0: B_t-B_0>\psi(B_t) \}$ a.s., for some measurable function $\psi$ (cf.\ Figure \ref{fig:rightborder}).
This is the unique solution to SEP of this form.
\end{proposition}

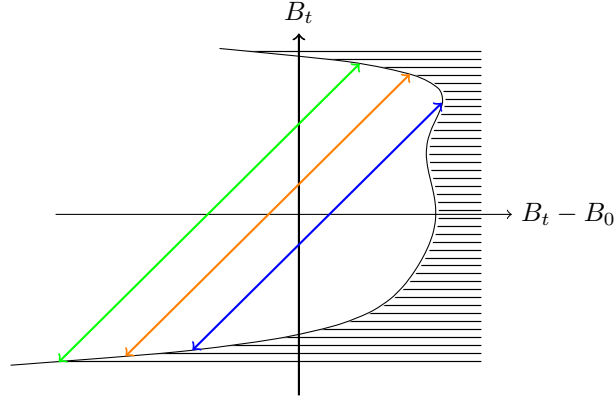
\begin{figure}
\begin{tikzpicture}[scale=0.4]
\draw[->] (-8,0) -- (7,0) node[right] {$B_t-B_0$};
\draw[->,thick] (0,-6) -- (0,6) node[above] {$B_t$};
\draw [smooth,variable=\y,domain=-5:5.5] plot ( {4.5+0.02*\y-0.2*\y*\y+0.02*\y*\y*\y+0.02*\y*\y*\y*\y+0.001*\y*\y*\y*\y*\y-0.001*\y*\y*\y*\y*\y*\y},{\y});
\fill [pattern = horizontal lines] (6,-5) -- plot [smooth,variable=\y,domain=-5:5.5] ( {4.6+0.02*\y-0.2*\y*\y+0.02*\y*\y*\y+0.02*\y*\y*\y*\y+0.001*\y*\y*\y*\y*\y-0.001*\y*\y*\y*\y*\y*\y},{\y}) -- (6,5.5) -- cycle;
\draw[<->,color=blue,thick] (-3.5,-4.5) -- (4.7,3.7); %
\draw[<->,color=orange,thick] (-5.7,-4.7) -- (3.65,4.65); %
\draw[<->,color=green,thick] (-7.9,-4.9) -- (2,5); %
\end{tikzpicture}
\caption{The Brownian motion travels along lines parallel to the main diagonal and is stopped
when it hits the boundary. }\label{fig:rightborder}
\end{figure}

\begin{figure}
\begin{tabular}{m{2.2in}m{0.5cm}m{2in}}
\begin{tikzpicture}[scale=0.3]
\draw[->] (-8,0) -- (7,0) node[above] {\footnotesize $B_t-B_0$};
\draw[->,thick] (0,-6) -- (0,6) node[above] {\footnotesize $B_t$};
\draw[-,color = blue,dashed] (0,1) -- (3,4);
\draw[-,color = blue,dashed] (0,2) -- (2,4);
\draw[<->,thick,color=blue] (2,3) -- (4,5);
\draw[fill,color=blue] (3,4) circle (1.5pt);
\draw[fill,color=blue] (2,4) circle (3pt);

\draw[-,color = orange,dashed] (0,-1) -- (-3,-4);
\draw[-,color = orange,dashed] (0,-2) -- (-2,-4);
\draw[<->,thick,color=orange] (-3,-5) -- (-1,-3);
\draw[fill,color=orange] (-3,-4) circle (3pt);
\draw[fill,color=orange] (-2,-4) circle (2pt);
\end{tikzpicture}
& {\large $\phantom{>}$} &
\begin{tikzpicture}[scale=0.3]
\draw[->] (-8,0) -- (7,0) node[above] {\footnotesize $B_t-B_0$};
\draw[->,thick] (0,-6) -- (0,6) node[above] {\footnotesize $B_t$};
\draw[-,color = blue,dashed] (0,1) -- (3,4);
\draw[-,color = blue,dashed] (0,2) -- (2,4);
\draw[<->,thick,color=blue] (1,3) -- (3,5);
\draw[fill,color=blue] (3,4) circle (3pt);
\draw[fill,color=blue] (2,4) circle (1.5pt);

\draw[-,color = orange,dashed] (0,-1) -- (-3,-4);
\draw[-,color = orange,dashed] (0,-2) -- (-2,-4);
\draw[<->,thick,color=orange] (-4,-5) -- (-2,-3);
\draw[fill,color=orange] (-3,-4) circle (2pt);
\draw[fill,color=orange] (-2,-4) circle (3pt);
\end{tikzpicture}
\end{tabular}
\caption{This picture shows two examples of a stop-go pair (one in blue, one in orange). For both pairs, the left-hand image
shows a pair that leads to higher costs than the right-hand pair for a Spence--Mirrlees type cost function.}\label{fig:sgpairs}
\end{figure}
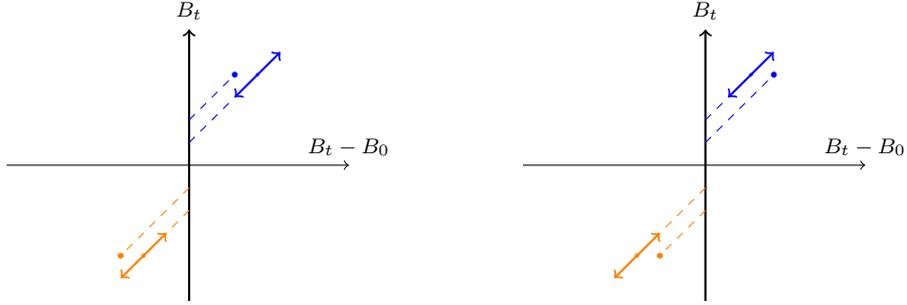

\begin{proof}
Pick $c$ such that \eqref{eq:OptSEP} is well posed. Since $c$ is continuous, hence lower-semicontinuous, there is a minimizer $\hat\tau$. Put $\gamma(g,t)=c(g(0),g(t))$. Pick by Theorem \ref{GlobalLocal3} a $\gamma$-monotone set $\Gamma$ supporting $\hat\tau.$
We claim that
\begin{equation}\label{eq:SGSpMi}
\SG \supseteq \{((f,s),(g,t)) \in S \times S : g(t) = f(s), g(0) > f(0) \}.
\end{equation}
This is represented in Figure \ref{fig:sgpairs}.
To see this, pick some stopping time $\sigma$ with positive and finite expectation and observe
that for some fixed $y \in \R$ the map
\begin{align}\label{eq:concave}
x \mapsto \E[c(x,y+B_\sigma)] - c(x,y) 
\end{align}
is strictly decreasing due to the strict concavity of $y \mapsto c_x(x,y)$.
Pick some pair $((f,s),(g,t)) \in S \times S$ satisfying $f(s)=g(t)$ and $g(0) > f(0)$. Then \eqref{eq:concave} implies
\[\E[c(f(0),f(s)+B_\sigma)] + c(g(0),g(t)) > c(f(0),f(s)) + \E[c(g(0),g(t)+B_\sigma)]. \]
Now define
\begin{align*}
 \mathcal R_{\op}&:=\{(d,y):\exists (g,t)\in\Gamma, g(t)=y, d>g(t)-g(0)\}\\
 \mathcal R_{\cl}&:=\{(d,y):\exists (g,t)\in\Gamma, g(t)=y, d\geq g(t)-g(0)\}.
\end{align*}
Fix $(g,t)\in\Gamma$. Then, $(g(t)-g(0),g(t))\in\mathcal R_\cl$. 
Suppose that $\inf\{s \in [0,t]: (g(s)-g(0),g(s)) \in \mathcal{R}_{\op}\} < t$. Then there exists $s<t$ such that $(f,s) := \left(g_{\llcorner [0,s]},s\right) \in \Gamma^{<}$ and $(f(s)-f(0), f(s)) \in \mathcal{R}_{\op}$. By definition of $\mathcal{R}_{\op}$, it follows that there exists another path $(k,u) \in \Gamma$ such that $k(u) = f(s)$ and $k(0) > f(0)$. But then $\big( (f,s), (k,u)\big) \in \SG\cap \big(\Gamma^<\times \Gamma\big)$ which cannot be the case. Hence, $(g,t) \in \Gamma$ implies
\begin{align}\label{eq:closed le open}
   \inf\{s \in [0,t]: &(g(s)-g(0),g(s)) \in \mathcal{R}_{\cl}\} \\
& \le t
  \le \inf\{s \in [0,t]: (g(s)-g(0),g(s)) \in \mathcal{R}_{\op}\}.\nonumber
\end{align}
Denote the  hitting times of $\mathcal R_\op$ and $\mathcal R_{\cl}$ by $\tau_{\op}$ and $\tau_{\cl}$ resp., i.e. 
$$ \tau_\op(\omega)=\inf\{ t \geq 0 : (B_t(\omega)-B_0(\omega),B_t(\omega))\in\mathcal{R}_\op\}\ ,$$
and correspondingly for $\tau_\cl$. Pick $\omega$ such that $(g,t)=((B_s)_{s\leq\hat\tau}(\omega),\hat\tau(\omega))\in\Gamma.$ It then follows from \eqref{eq:closed le open} that
$$\tau_\cl(\omega)\leq\hat\tau(\omega)\leq\tau_{\op}(\omega)\ .$$
Since $\P(((B_s)_{s\leq\hat\tau}(\omega),\hat\tau(\omega))\in\Gamma)=1$ we have $\tau_\cl\leq\hat\tau\leq\tau_{\op}$ a.s. It remains to show that $\tau_\cl=\tau_\op$ a.s. Put $\psi(y):=\inf\{d : (d,y)\in\mathcal R_\cl\}$ and observe that $\tau_\cl=\inf\{t\geq 0:B_t-B_0\geq\psi(B_t)\}$ and $\tau_\op=\inf\{t\geq 0 : B_t-B_0> \psi(B_t)\}.$ (Note that $\psi$ is in general only lower semianalytic, and in particular universally measurable. We can therefore replace $\psi$ by a Borel measurable function that agrees with $\psi$ outside a set $N$ such that
$\P[B_{\tau_\op}\in N] = \P[B_{\tau_\cl}\in N] = 0$ and
is $\psi(x) \equiv +\infty$ on $N$. We denote this function again by $\psi$, as the transformation does not affect $\tau_{\op/\cl}$.) To show that $\tau_\op=\tau_\cl$ it is easier to work in a slightly deformed geometry. Using the transformation $(d,y)\mapsto (d-y,y)$ and setting
\begin{align*}
\tilde{\mathcal R}_\op &= \{(x,y) : x>\tilde\psi(y):=\psi(y)-y\}\\
\tilde{\mathcal R}_\cl &= \{(x,y) : x\ge\tilde\psi(y):=\psi(y)-y\}
\end{align*}
it is immediate that the respective hitting times $\tilde\tau_\op$ and $\tilde\tau_\cl$ of the process $(-B_0,B_t)$ agree with $\tau_\op$ and $\tau_\cl$. Their almost sure equality follows from Lemma \ref{lem:barrierequality} below.
% 
% 
% Since the set 
% $$\{\tilde\psi(y):y\text{ is a local minimum}\}$$
% is at most countable we have that $\tilde\tau_\op=0 \Leftrightarrow \tilde\tau_\cl=0$ almost surely since $\mu$ does not charge atoms. Hence, we can assume that $\tau_\cl,\tau_\op>0$ a.s. Approximating $\tilde{\mathcal R}_\cl$ from outside by sets of the form $\cup_{i=1}^N [a_i,\infty)\times I_i$ for some intervals $I_i$ and from inside by sets of the form $\cup_{i=1}^n [b_n,\infty)\times \{y_n\}$ we can further sandwich $\tilde\tau_\cl\leq\tilde\tau_\op$ between the hitting times $\underline\tau\leq\tilde\tau_\cl$
% and $\overline\tau\geq\tilde\tau_\op$ of these sets. Since $\P[\overline\tau-\underline\tau>\delta]$ can be made arbitrarily small by making the approximation of $\mathcal R_\cl$ finer and finer the result follows.

Uniqueness of this embedding can be shown by an argument due to Loynes \cite{Lo70} based on previous work by Root \cite{Ro69} for which we refer to \cite{BeHeTo15}.
\end{proof}

\begin{lemma}\label{lem:barrierequality}
Let $\pi$ be a probability measure on $\R^2$ such that the first marginal $\pi_1$ is continuous (i.e. atom free) and let $\psi: \R\to \R$ be a Borel function. Define subsets
$$R_\op:= \{(x,y): x>\psi(y)\},\quad R_\cl:= \{(x,y): x\geq\psi(y)\}.$$
Start a vertically moving Brownian motion $B$ in $\pi$ and define
$$ \tau_\op(x,y):= \inf\{ t\geq 0: (x, y+B_t)\in R_\op\}, \quad \tau_\cl(x,y):= \inf\{ t\geq 0: (x, y+B_t)\in R_\cl\}.$$
Let $(X,Y)$ have distribution $\pi$ and be independent from
$B$, then $\tau_\cl(X,Y)=\tau_\op(X,Y)$ almost surely.
\end{lemma}
\begin{proof}
%We only sketch the proof.
As $R_\op \subseteq R_\cl$, we have $\tau_\cl(x,y)\leq \tau_\op(x,y)$. 

 We say that $y$ is a local minimum of $\psi$ if $\psi(y')\geq \psi(y)$ for all $y'$ in a neighbourhood of  $y$. Set 
$$I:=\{\psi (y): y \text{ is a local minimum of } \psi\}.$$
It then follows that $I$ is at most countable: assume by contradiction that there exist an uncountable family $A\subseteq \R$ and corresponding neighborhoods $(a-\eps_a, a+\eps_a), a\in A$ such that for $x\in (a-\eps_a, a+\eps_a)$ we have $\psi(x)\geq \psi(a)$  and $a,a'\in A, a\neq a'$ implies $\psi(a)\neq \psi(a')$. Passing to an uncountable subset of $A$, we can assume that there is some $\eta >0$ such that $\eps_a>\eta$ for all $a\in A$. For $a\neq a'$ both in $A$ we cannot have $|a-a'| < \eta$ for then  $a\in (a'-\eps_{a'}, a'+\eps_{a'})$ as well as $a'\in (a-\eps_{a}, a+\eps_{a})$ which would imply that $\psi(a)=\psi(a')$. Hence $|a-a'|\geq \eta$ which implies that $A$ is countable, giving a contradiction. 

\smallskip

On the complement of $I\times \R$ we have almost surely
\begin{align}\tau_\op(x,y)=0 \quad \Longleftrightarrow \quad \tau_\cl(x,y)=0\end{align}
and since $I$ is countable and $\pi_1$ is continuous we have $\pi(I \times \R) = 0$, we can
conclude that $\{\tau_\op(X,Y) = 0 \} \Delta \{\tau_\cl(X,Y) = 0\}$ is
a nullset. 

Now consider $x,y \in I^c \times \R$ such that $\tau_\cl(x,y),\tau_\op(x,y) > 0$ and fix $\eps>0$. 
Pick $\eta>0$ such that $$\P(\underbrace{\tau_\cl>\eta, |B_{\tau_\cl} - y| > \eta}_{=:G})> 1-\eps.$$
Fix a family of intervals $(I_n)_n$ such that they form a partition of the vertical axis. For each $n$ define $x_n := \inf\{\psi(y):y\in I_n\}$ and set $a_n := x_n - \delta$
and $b_n := x_n + \delta$ for some $\delta >0$.
Fix $y_n\in I_n$ such that $[b_n,\infty)\times \{y_n\}\subset R_\op$. We then define stopping times 
\begin{align}\underline \tau&:= \inf\left\{t \ge 0: (x,y+B_t) \in \bigcup_{n: I_n\cap (y-\eta, y+\eta)=\emptyset}  [a_n,\infty)\times I_{n}\right\}\\
\overline \tau&:= \inf\left\{t \ge 0: (x,y+B_t) \in \bigcup_{n: I_n\cap (y-\eta, y+\eta)=\emptyset}  [b_n,\infty)\times \{y_{n}\} \right\}. 
\end{align} 
On the set $G$ we then have $\underline \tau \leq \tau_\cl\leq \tau_\op\leq \overline \tau$. Choosing the partition sufficiently fine, and $\delta$ sufficiently small, we obtain $\P(\overline \tau-\underline \tau >\eps) <\eps$, proving the assertion.
\end{proof}

\begin{remark}\label{rem:leftmonotone}
Define $\Xi:=\{(x,y):\exists (g,t)\in\Gamma, g(0)=x,g(t)=y\}$. By the previous theorem, this is precisely a supporting set of the martingale coupling $(B_0,B_{\tau_{BJ}})$. Moreover, $\tau_{BJ}$ being the hitting time of a right barrier $\mathcal R$ (i.e.\ $(d,y)\in\mathcal R, d'>d \implies (d',x)\in\mathcal R$) in the $(B_t-B_0,B_t)$-phase space it immediately follows---one can directly read it off from Figure \ref{fig:rightborder}---that $\Xi$ is left monotone (as defined in \cite{BeJu16}), i.e.\ if $(x_1,y_1),(x_1,y_2),(x_2,y') \in \Xi$ such that $x_1 < x_2$ and $y_1<y_2$ then $y'\notin(y_1,y_2)$.
\end{remark}

\begin{remark}
 Note that to deduce the left monotone structure of an optimal martingale coupling the assumption of $\mu$ not charging atoms is not needed. In fact, it is sufficient to identify the stop-go pairs $\SG$ as in  \eqref{eq:SGSpMi} since observing a path $g$ being stopped at $g(t)=y'>g(0)=x_2$ no path starting at $x_1<x_2$ can cross the level $y'$ not being stopped without violating \eqref{eq:SGSpMi}. A similar remark applies to the embedding minimizing $\pm|x-y|$ below.
\end{remark}

\begin{remark}\label{rem:leftbarrier}
 Assuming in Proposition \ref{prop:SpMi} $c_{xyy}>0$ we can run the very same proof to show the existence of a unique solution $\tau_{BJ}'$ of SEP minimizing $E[c(B_0,B_\tau)]$ and which is of the form $\tau_{BJ}' = \inf\{t > 0: B_t-B_0<\psi'(B_t) \}$ a.s., for some measurable function $\psi'$. This solution corresponds to the hitting time of a left barrier $\mathcal R'$ (i.e.\ $(d,y)\in\mathcal R', d'<d \implies (d',x)\in\mathcal R'$)
\end{remark}

Similarly we can prove the well-known properties for the optimizers in the so-called forward start straddle problem:

\begin{proposition}
Assume that $\mu(\{x\})=0$ for all $x\in\R.$ There exists a stopping time $\tau_{HN}$ which minimizes
\[\E[-|B_\tau-B_0|]\]
over all solutions to (SEP) and which is of the form 
$\tau_{HN} = \inf\{t > 0: B_t-B_0 \notin (\psi_1(B_t),\psi_2(B_t))\}$ a.s.,
for some measurable functions  $\psi_1\leq 0 \leq \psi_2$.
\end{proposition}

\begin{proof}
The proof runs along similar lines as for Proposition \ref{prop:SpMi} once we identified the secondary stop-go pairs. For simplicity we assume that $\nu$ has finite third moment.
Consider the functionals
$\gamma((f,s)) = -|f(s)-f(0)|$ as well as $\tilde{\gamma}((f,s)) =|f(s)-f(0)|^3$.

We claim that
\begin{equation}
\SG_2 \supseteq \{((f,s),(g,t)) \in S \times S : g(t) = f(s),  0 < (g(0)-f(0))(g(t)-g(0)) \}.
\end{equation}

Fix some $y \in \R$ and a stopping time $\sigma$ with positive and finite expectation and observe that
\[ x \mapsto \E[|y+B_{\sigma}-x|] - |y-x| \]
is monotonically increasing on $(-\infty,y]$ and monotonically decreasing
on $(y,+\infty)$. 
% Moreover, the map
% \[ x \mapsto \E[-(y+B_{\sigma}-x)^3] + (y-x)^3 \]
% is strictly increasing on $(-\infty,y]$ and strictly decreasing on $(y,+\infty)$.
Pick some pair $((f,s),(g,t)) \in S\times S$ satisfying $f(0) < g(0) < g(t)=f(s)$. Then,
\begin{align*}
\E[|f(s)+B_\sigma-&f(0)|] + |g(t) - g(0)| \\
 &\leq |f(s)-f(0)| + \E[|g(t)+B_\sigma-g(0)|],
% , \text{ and} \\
% \E[-(f(s)+B_\sigma&-f(0))^3] - (g(t) - g(0))^3 \\
% &> -(f(s)-f(0))^3 - \E[(g(t)+B_\sigma-g(0))^3].
\end{align*}
with equality iff $B_\sigma\geq -(g(t)-g(0)).$ 
Then we have $|g(t)+B_\sigma-g(0)|^3=(g(t)+B_\sigma-g(0))^3$ as well as $|f(s)+B_\sigma-f(0)|^3=(f(s)+B_\sigma-f(0))^3$. Hence, by strict concavity of $-3(y-x)^2=\partial_x |x-y|^3$ on the set $\{(x,y):x<y\}$ we have
\begin{align*}
 \E[(f(s)+B_\sigma&-f(0))^3] + (g(t) - g(0))^3 \\
 &> (f(s)-f(0))^3 +\E[(g(t)+B_\sigma-g(0))^3],
\end{align*}
showing that $((f,s),(g,t))\in\SG_2$. An almost identical analysis shows the claim for the case  $f(s)=g(t) < g(0) < f(0)$ finishing the proof.
\end{proof}

\begin{figure}
\begin{tikzpicture}[scale=0.4]
\draw[->] (-6,0) -- (6,0) node[right] {$B_t-B_0$};
\draw[->,thick] (0,-6) -- (0,6) node[above] {$B_t$};

\draw [smooth,variable=\y,domain=-5:5.5] plot ( {3.5+0.01*\y-0.05*\y*\y+0.0025*\y*\y*\y+0.00125*\y*\y*\y*\y+0.00003125*\y*\y*\y*\y*\y-0.000015625*\y*\y*\y*\y*\y*\y},{\y});
\draw [smooth,variable=\y,domain=-5:5.5] plot ( {-2.8+0.01*\y-0.05*\y*\y+0.0025*\y*\y*\y+0.00125*\y*\y*\y*\y+0.00003125*\y*\y*\y*\y*\y-0.000015625*\y*\y*\y*\y*\y*\y},{\y});

\fill [pattern = horizontal lines] (5,-5) -- plot [smooth,variable=\y,domain=-5:5.5] ( {3.6+0.01*\y-0.05*\y*\y+0.0025*\y*\y*\y+0.00125*\y*\y*\y*\y+0.00003125*\y*\y*\y*\y*\y-0.000015625*\y*\y*\y*\y*\y*\y},{\y}) -- (5,5.5) -- cycle;

\fill [pattern = horizontal lines] (-5,-5) -- plot [smooth,variable=\y,domain=-5:5.5] ( {-2.9+0.01*\y-0.05*\y*\y+0.0025*\y*\y*\y+0.00125*\y*\y*\y*\y+0.00003125*\y*\y*\y*\y*\y-0.000015625*\y*\y*\y*\y*\y*\y},{\y}) -- (-5,5.5) -- cycle;

\draw[<->,color=blue,thick] (-3.9,-4.9) -- (3.3,2.3); %
\draw[<->,color=orange,thick] (-3.1,-2.1) -- (3.2,4.2); %
\draw[<-,color=green,thick] (-2.8,0.2) -- (2.5,5.5); %
\end{tikzpicture}
\caption{The barrier solution optimizing the SEP for the cost functional $-|B_t-B_0|$ in the phase space
$(B_t-B_0,B_t)$.}\label{fig:neubergerbarrier}
\end{figure}
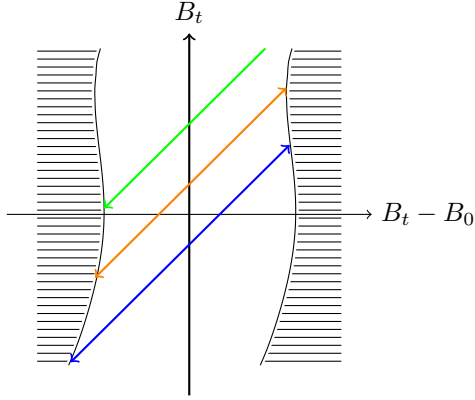

Looking at Figure \ref{fig:neubergerbarrier} one directly sees that the martingale coupling $(B_0,B_{\tau_{HN}})$ is concentrated on the graph of two non-decreasing functions.

Furthermore, changing a few signs at the appropriate spaces in the last proof we directly get:

\begin{proposition}
Suppose the supports of the marginal measures $\mu$ and $\nu$ are disjoint and $\mu(\{x\})=0$ for all $x\in\R$. 
There exists a stopping time $\tau_{HK}$ which minimizes
\[\E[|B_\tau-B_0|] \]
over all solutions to (SEP) and which is of the form
$\tau_{HK} = \inf\{t > 0 : B_t - B_0 \in (\psi_1(B_t),\psi_2(B_t))$ a.s.,
for some measurable functions  $\psi_1 \leq 0 \leq \psi_2$.
Suppose additionally the existence of an interval $I$ such that $\mu(I)=\nu(I^c)=1$ (cf.\ \cite[Dispersion Assumption 2.1]{HoKl15}). Then this
is the unique solution to (SEP) of this form. 
\end{proposition}

\begin{figure}
\begin{tikzpicture}[scale=0.4]
\draw[->] (-6,0) -- (6,0) node[right] {$B_t-B_0$};
\draw[->,thick] (0,-6) -- (0,6) node[above] {$B_t$};
\draw [smooth,variable=\y,domain=2.05:5.7] plot ( {(\y - 2)^(3/2)},{\y});
\draw [smooth,variable=\y,domain=-5:-2.05] plot ( {-abs(\y + 2)^2},{\y});

\fill [pattern = horizontal lines] (0.1,5.7) -- plot [smooth,variable=\y,domain=2.1:5.7] ({(\y - 2)^(3/2)-0.1},{\y}) -- cycle;
\fill [pattern = horizontal lines] (-0.1,-5) -- plot [smooth,variable=\y,domain=-5:-2.1] ({-abs(\y + 2)^2+0.1},{\y}) -- cycle;

\draw[<->,color=blue,thick] (-2.65,-3.65) -- (6.5,5.5); %
\draw[<->,color=orange,thick] (-5.25,-4.25) -- (3.2,4.2); %
\draw[fill,color=green] (0,3) circle (3pt);
\end{tikzpicture}
\caption{The barrier solution for the cost functional $|B_t - B_0|$ in the phase space $(B_t-B_0,B_t)$.}
\label{fig:klimmekbarrier}
\end{figure}
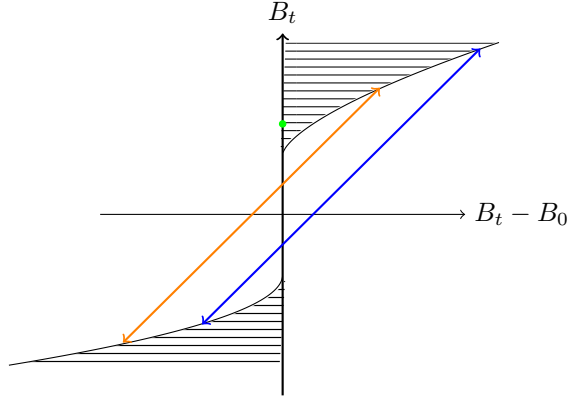

Looking at Figure \ref{fig:klimmekbarrier} we see that the martingale coupling $(B_0,B_{\tau_{HK}})$ is concentrated on the graph of two non-increasing functions.

\begin{remark}
 The last proposition remains true also if the marginals are not disjoint but the full Dispersion Assumption of \cite{HoKl15} holds. In fact, it then follows from the structure of $\SSG$ that the common mass $\mu\wedge \nu$ stays where it is. This is similar to the Rost embedding with general starting law, e.g.\ \cite[Theorem 2.4]{BeCoHu16}.

 Furthermore the result holds true for the more general class of cost functions $c(x,y) = |x-y|^p$ with $0 < p \leq 1$. Writing
 up the condition for the stop-go pairs shows that the stop-go pairs agree with the ones for $|x-y|$.
\end{remark}

\begin{remark}
The proof of uniqueness of the solution under the strong dispersion assumption $\mu(I) = \nu(I^c) = 1$ is similar to Loynes' argument for hitting times of a single barrier. Indeed, if $I=[a,b]$ then there is an 'upper' part of the barrier responsible for embedding $\nu_{\upharpoonright [b,\infty)}$ and a 'lower' part of the barrier responsible for embedding $\nu_{\upharpoonright (-\infty,a]}$ which do not interfere with each other. On each of these barrier one can run Loynes' argument.
\end{remark}

\section{Monotone Transformations}\label{sec:trafo}

We propose a class of transformations of cost functions that will in particular
comprise the transformation in \cite{CaLaMa16} and linear transformations from
$h(x-y)$ to $h(x+y)$. The idea is that this allows us to derive monotone sets for the
transformed cost function as transformations of monotone sets of the original
cost function. To prove this monotonicity, we need an accompanying transformation of
measures that preserves competitors.

\subsection{Competitor-Preserving Transformations}

Let us start with the definition of the types of transformations we are interested in (where $\cE(X)$ denotes the finite
measures on a space $X$):

\begin{definition}
We call a bijective transformation $\tau: \cE(I \times J) \to \cE(I' \times J')$ 
\emph{competitor preserving}, if
for given competitors $\alpha'$ and $\beta'$ concentrated on $I' \times J'$ we have that 
$\tau^{-1}(\alpha')$ and $\tau^{-1}(\beta')$ are also competitors.
\end{definition}

We will identify competitor preserving $\tau$ of a specific form:
\begin{definition}\label{def:monotonetrafo}
Let $T: I \times J \to I' \times J'$ be a bijective map and $h: I \times J \to \R^+$. We say the pair $(T,h)$ is a  \emph{monotonicity preserving transformation} if
 the induced transformation $\tau: \cE(I \times J) \to \cE(I' \times J')$ 
defined via
\[\int g \, d\tau(\alpha) = \int (g \circ T) \cdot h \, d\alpha\]
for all bounded continuous $g$ is competitor preserving.
\end{definition}

The inverse transformation of $\tau$ from Definition \ref{def:monotonetrafo} is then
characterized by
\[ \int g \, d\tau^{-1}(\alpha') = \int \frac{g \circ T^{-1}}{h \circ T^{-1}} \, d\alpha'\]
for all bounded continuous functions $g$. This fact will be used in the proof of the following theorem,
characterizing the competitor-preserving transformations $(T,h)$.
\begin{theorem}
\label{thm:compensatablerep}
Let $T(x,y) = (s(x,y),t(x,y))$ be a bijective map and $h: I \times J \to (0,\infty)$.
Then we must have that $s$ is constant in $y$ and $t$ and $h$ are constant in $x$. Furthermore $t$ and $h$ must be of
one of the following forms:
\begin{enumerate}[(1)]
\item $h$ is constant and $t$ is affine in $y$;
\item $h$ is affine in $y$ ($h(y) = c(y - b)$) and $t$ is of the form $t(y) = a/(y-b)$.
\end{enumerate}
\end{theorem}

The proof of Theorem \ref{thm:compensatablerep} is structured in five preliminary results. The first characterizes functions that cannot distinguish between competitors. The next three describe restrictions on the class of bijective monotonicity preserving transformations arising from their
definition. After that we will  prove
a representation of the inverse transformation of $(T,h)$ in Proposition \ref{prop:invcomprep} and derive the representation of $T$ from that.

\begin{lemma}\label{lem:compchar}
Suppose $f: I \times J \to \R$ for $I$ and $J$ (possibly unbounded) intervals in $\R$ is such that
$\alpha(f) = \beta(f)$ whenever $\alpha$ and $\beta$ are competitors. Then we can find a representation of $f$ of the
form $f(x,y) = \phi(x) + \psi(y) + k(x)y$.
\end{lemma}

\begin{proof}
Consider the competitors 
\begin{align*} 
\alpha &:= \lambda \delta_{(x_1,y_1)} + (1-\lambda) \delta_{(x_1,y_2)} + \delta_{(x_2,y_\lambda)} \\
\beta &:= \lambda \delta_{(x_2,y_1)} + (1-\lambda) \delta_{(x_2,y_2)} + \delta_{(x_1,y_\lambda)}
\end{align*}
for $y_1 < y_\lambda < y_2 \in J$ and $x_1,x_2 \in I$ arbitrary and $\lambda \in (0,1)$ such that 
$y_\lambda = \lambda y_1 + (1-\lambda)y_2$. Then by assumption we must have $\alpha(f) = \beta(f)$
which amounts to
\begin{align*} 
\lambda f(x_1,y_1) + &(1-\lambda) f(x_1,y_2) - f(x_1,y_\lambda) \\
&= \lambda f(x_2,y_1) + (1-\lambda) f(x_2,y_2) - f(x_2,y_\lambda).
\end{align*}
As $y_1 < y_\lambda < y_2$ were arbitrary, the non-linear shape of $f$ along $y$ does not depend on $x$ and we can set
$g(x,y) := f(x,y) - \psi(y)$ for $\psi(y) := f(x_0,y)$ and $x_0 \in I$ arbitrary but fixed.
For any $x \in I$ we then have that $g(x,y)$ is linear in $y$ and the result follows.
\end{proof}

Any function $f$ satisfying the assumptions of the last Lemma will be called \emph{competitorblind.}

Observe that two measures $\alpha$ and $\beta$ are competitors if and only if they cannot distinguish
between competitorblind functions. More concretely, they are competitors if for any 
$g \in \mathcal{C}_b(\R)$ we have
\begin{align*}
\int g(x) d\alpha(x,y) &= \int g(x) d\beta(x,y) &\quad (C1)\\
\int g(y) d\alpha(x,y) &= \int g(y) d\beta(x,y) &\quad (C2)\\
\int g(x)(y-x) d\alpha(x,y) &= \int g(x)(y-x) d\beta(x,y) &\quad (C3)
\end{align*}
and sums of these functions (and their limits) are the only functions such that equality
holds for arbitrary competitors.

\begin{lemma}
\label{lem:hrestrictions}
Let $\tilde s: I' \to I$ and $\tilde t: J' \to J$ be non-constant functions and $\tilde h: I' \times J' \to (0,\infty)$ 
be a competitorblind function.
\begin{enumerate}[(i)]
\item If $g(\tilde s(x))\tilde h(x,y)$ is competitorblind for any $g \in \mathcal{C}_b(\R)$, then $\tilde h$ is linear in $y$.
\item If $g(\tilde t(y))\tilde h(x,y)$ is competitorblind for any $g \in \mathcal{C}_b(\R)$, then $\tilde h$ is constant in $x$.
\end{enumerate}
\end{lemma}

\begin{proof}
We can write $\tilde h(x,y) = \phi(x) + \psi(y) + k(x)y$ by assumption. In the case of (i), we always have that
$g(\tilde s(x))(\phi(x) + k(x)y)$ is competitorblind. Therefore we must in particular have that
$g(\tilde s(x))\psi(y)$ is competitorblind. For arbitrary $x_1 < x_2$ and $y_1 < y_\lambda < y_2$ with $\lambda \in (0,1)$ such that
$y_\lambda = \lambda y_1 + (1-\lambda) y_2$ we then have
\[(g(\tilde s(x_1)) - g(\tilde s(x_2)))(\lambda \psi(y_1) + (1-\lambda) \psi(y_2) - \psi(y_\lambda)) = 0 \]
by choosing as above a pair of competitors
\begin{align*} 
\alpha &:= \lambda \delta_{(x_1,y_1)} + (1-\lambda) \delta_{(x_1,y_2)} + \delta_{(x_2,y_\lambda)} \\
\beta &:= \lambda \delta_{(x_2,y_1)} + (1-\lambda) \delta_{(x_2,y_2)} + \delta_{(x_1,y_\lambda)}.
\end{align*}
As $\tilde s$ was assumed to be non-constant, we can choose $x_1$ and $x_2$ such that $\tilde s(x_1) \neq \tilde s(x_2)$ and
therefore $\psi$ has to be linear in $y$ and then also $\tilde h$ has to be linear in $y$.

In the case of (ii), we always have that $g(\tilde t(y))\psi(y)$ is competitorblind and hence 
$g(\tilde t(y))(\phi(x) + k(x)y)$ must be competitorblind. As in part (i) we derive that for any
$x_1 < x_2 \in I$ and $y_1 < y_\lambda < y_2 \in J$ with $\lambda \in (0,1)$ accordingly, we have 
that
\begin{align*}
(\phi&(x_1) - \phi(x_2))[\lambda g(\tilde t(y_1)) + (1-\lambda)g(\tilde t(y_2)) - g(\tilde t(y_\lambda))] \\
&+ (k(x_1) - k(x_2))[\lambda g(\tilde t(y_1))y_1 + (1-\lambda)g(\tilde t(y_2))y_2 - g(\tilde t(y_\lambda))y_\lambda] = 0.
\end{align*}
As $\tilde t$ is non-constant, we can find $y_1$ and $y_2$ with $\tilde t(y_1) \neq \tilde t(y_2)$ and can then
choose $g$ such that $\lambda g(\tilde t(y_1)) + (1-\lambda)g(\tilde t(y_2)) - g(\tilde t(y_\lambda)) = 0$ but
$\lambda g(\tilde t(y_1))y_1 + (1-\lambda)g(\tilde t(y_2))y_2 - g(\tilde t(y_\lambda))y_\lambda \neq 0$ which yields that $k$ is a constant
function. Now choosing a function $g$ such that $\lambda g(\tilde t(y_1)) + (1-\lambda)g(\tilde t(y_2)) - g(\tilde t(y_\lambda)) \neq 0$
yields that also $\phi$ is a constant function.
\end{proof}

The proof of the following Lemma is elementary but technical and therefore deferred to the appendix.

\begin{lemma}
\label{lem:tdependence}
Assume we have a function $\tilde t(x,y) := \frac{a(x) + b(x)y}{c(x)+d(x)y}$ defined on a rectangle $I' \times J'$ such that
$\tilde h(x,y)=c(x)+d(x)y > 0$ everywhere and $\tilde t(x,y)^k(c(x) + d(x)y)$ is competitorblind for $k = 2,3$. If $\tilde t$ is non-constant in $y$ for
all $x \in I'$ then $\tilde t$ is constant in $x$. Moreover, we have for some $\bar a,\bar b,\bar c,\bar d\in\R$
\begin{itemize}
 \item either $\tilde h(x,y)=\bar dy$ and $\tilde t(x,y)=\frac{\bar a}{\bar d y}+\frac{\bar b}{\bar d}$
\item or $\tilde h(x,y)=\bar c$ and $\tilde t(x,y)=\frac{\bar a+\bar by}{\bar c}$.
\end{itemize}
\end{lemma}

Using this Lemma we can prove the main part of Theorem \ref{thm:compensatablerep}. It is easier to prove the shape of
the inverse transform first, which will be done in Proposition \ref{prop:invcomprep}, and deduce the representation of the original transform from it.

We first prove another auxiliary result.

\begin{lemma}
\label{lem:compBlindInjective}
Let $t: I \times J \to \tilde{J}$, and $h: I \times J \to (0,\infty)$ be such
that $(g \circ t)\cdot h$ is competitorblind for all $g \in \C_b(\R)$, then
$t$ cannot be injective.
\end{lemma}

\begin{proof}
Assume the contrary and consider the usual pair of competitors
\begin{align*} 
\alpha &:= \lambda \delta_{(x_1,y_1)} + (1-\lambda) \delta_{(x_1,y_2)} + \delta_{(x_2,y_\lambda)} \\
\beta &:= \lambda \delta_{(x_2,y_1)} + (1-\lambda) \delta_{(x_2,y_2)} + \delta_{(x_1,y_\lambda)}.
\end{align*}
for some $x_1,x_2 \in I$ and $y_1 < y_\lambda < y_2 \in J$. 
Then for arbitrary $g \in \C_b(\R)$ we must have
\begin{align*}
\lambda &g(t(x_1,y_1))h(x_1,y_1) + (1-\lambda)g(t(x_1,y_2))h(x_1,y_2) - g(t(x_1,y_\lambda))h(x_1,y_\lambda) \\
&= \lambda g(t(x_2,y_1))h(x_2,y_1) + (1-\lambda)g(t(x_2,y_2))h(x_2,y_2) - g(t(x_2,y_\lambda)) h(x_2,y_\lambda)
\end{align*}
By injectivity of $t$ we can consider a function $g$ such that $g(t(x_i,y_j)) = 1/h(x_i,y_j)$
for $i,j \in \{1,2\}$. For any such $g$ we must have
\[g(t(x_1,y_\lambda))h(x_1,y_\lambda) = g(t(x_2,y_\lambda)) h(x_2,y_\lambda).\]
Now we can easily choose values of $g$ that contradict this property.
\end{proof}

\begin{proposition}
\label{prop:invcomprep}
Let $\tilde s: I' \times J' \to I$ and $\tilde t: I' \times J' \to J$ be such that $\tilde T(x,y) = (\tilde s(x,y),\tilde t(x,y))$ maps bijectively
from $I' \times J'$ to $I \times J$. Furthermore let $\tilde h: I' \times J' \to (0,\infty)$ be chosen such that
for any $g \in \C_b(\R)$ we have that $(g \circ \tilde s) \cdot \tilde h, (g \circ \tilde t) \cdot \tilde h$ and $(g \circ \tilde s)\cdot(\tilde t - \tilde s)\cdot \tilde h$
are competitorblind. Then $\tilde s$ is constant in $y$ and $\tilde t$ is constant in $x$.
\end{proposition}

\begin{proof}
First assume that $\tilde s$ is not constant in $y$. Accordingly we fix some $x_1 \in I'$ and $y_1 < y_2 \in J'$ such that $\tilde s(x_1,y_1) \neq \tilde s(x_1,y_2)$.

Recall that $\tilde t$ cannot be injective on $I' \times (y_1,y_2)$ by Lemma \ref{lem:compBlindInjective}.
Therefore $\tilde s$ cannot be constant on this area, because that would contradict the bijectivity
of $\tilde T$.

We now want to show that $\tilde s$ has to depend on $x$.
Assume for contradiction that $\tilde s$ does not depend on $x$. By the above paragraph it then cannot be constant in $y$ as well  on $I' \times (y_1,y_2)$.
This puts us in the setting of Lemma \ref{lem:hrestrictions} (ii) which implies that
$\tilde h$ has to be constant in $x$ on this interval. With slight abuse of notation we will therefore write $\tilde s(y) = \tilde s(x,y)$ and $\tilde h(y) = \tilde h(x,y)$
for $(x,y) \in I' \times (y_1,y_2)$. By assumption we have that $g(\tilde s(y))(\tilde t(x,y) - \tilde s(y))\tilde h(y)$ is competitorblind and
in particular $g(\tilde s(y))\tilde h(y)\tilde t(x,y)$ is also competitorblind on $I' \times (y_1,y_2)$. Another application of Lemma
\ref{lem:hrestrictions} (ii) yields that $\tilde h(y)\tilde t(x,y)$ does not depend on $x$. In particular $\tilde t$ does not depend on $x$ and therefore $\tilde T$ does not
depend on $x$ and cannot be bijective.

As we now know that $\tilde s$ cannot be constant in $x$, we can find $y_\lambda \in (y_1,y_2)$ and 
$x_2 \in I'$ such that $\tilde s(x_1,y_\lambda) \neq \tilde s(x_2,y_\lambda)$.

We consider again the usual pair of competitors
\begin{align*} 
\alpha &:= \lambda \delta_{(x_1,y_1)} + (1-\lambda) \delta_{(x_1,y_2)} + \delta_{(x_2,y_\lambda)} \\
\beta &:= \lambda \delta_{(x_2,y_1)} + (1-\lambda) \delta_{(x_2,y_2)} + \delta_{(x_1,y_\lambda)}.
\end{align*}
By assumption we have that for arbitrary $g \in \C_b(\R)$ the equality
\begin{align*}
\lambda &g(\tilde s(x_1,y_1))\tilde h(x_1,y_1) + (1-\lambda)g(\tilde s(x_1,y_2))\tilde h(x_1,y_2) - g(\tilde s(x_1,y_\lambda))\tilde h(x_1,y_\lambda) \\
&= \lambda g(\tilde s(x_2,y_1))\tilde h(x_2,y_1) + (1-\lambda)g(\tilde s(x_2,y_2))\tilde h(x_2,y_2) - g(\tilde s(x_2,y_\lambda))\tilde h(x_2,y_\lambda)
\end{align*}
holds. As in Lemma \ref{lem:compBlindInjective}, there is enough 
freedom in the choice of $g$ that this leads to a contradiction. 
Hence, $\tilde s$ does not depend on $y$ and we can write $\tilde s(x,y)=\tilde s(x)$.

Now Lemma \ref{lem:hrestrictions} (i) implies
that $\tilde h$ is affine in $y$. This yields that $g(\tilde s(x))\tilde s(x)\tilde h(x,y)$ is competitorblind for any continuous bounded $g$. As we assumed that
 $g(\tilde s(x))(\tilde t(x,y) - \tilde s(x))\tilde h(x,y)$ is competitorblind we also obtain that $g(\tilde s(x)) \tilde t(x,y) \tilde h(x,y)$ has to be competitorblind. 
This implies that $\tilde t(x,y)\tilde h(x,y)$ is affine in $y$ using Lemma \ref{lem:hrestrictions} (i) again. 

Therefore we can write $\tilde t(x,y) = \frac{a(x) + b(x)y}{c(x)+d(x)y}$
for some functions $a,b,c$ and $d$, where $\tilde h(x,y) = c(x)+d(x)y$ and therefore $c(x)+d(x)y > 0$ everywhere.
By assumption we have that $g(\tilde t(x,y))\tilde h(x,y)$ is competitorblind for bounded continuous $g$ and by approximation then also that
$\tilde t(x,y)^k\tilde h(x,y)$ is competitorblind. We must also have that for all $x$, the function $y \mapsto \tilde t(x,y)$ is
non-constant, because otherwise $\tilde T(x,y) = (\tilde s(x),\tilde t(x,y))$ is not injective. Therefore $\tilde t$ and $\tilde h$ satisfy the
assumptions of Lemma \ref{lem:tdependence} which yields that $\tilde t$ is constant in $x$.
\end{proof}

\begin{proof}[Proof of Theorem \ref{thm:compensatablerep}]
Let $\alpha',\beta'$ be competitors concentrated on $I' \times J'$. 
We need that $\alpha = \tau^{-1}(\alpha')$ and $\beta = \tau^{-1}(\beta')$ for $\tau$ as defined above are also competitors.
By the definition of $\tau$ we have for a function $f:I \times J$ that 
$\int f(x,y) d\alpha = \int \frac{f(T^{-1}(x',y'))}{h(T^{-1}(x',y'))} d\alpha'$. Set $\tilde h(x',y') := \frac{1}{h(T^{-1}(x',y'))}$
and $T^{-1}(x',y') = (\tilde s(x',y'),\tilde t(x',y'))$.
For $\alpha$ and $\beta$ to be competitors, we need that conditions $(C1)-(C3)$ hold (introduced above Lemma \ref{lem:hrestrictions}). This gives
\begin{align*}
\int g(x) d\alpha &= \int g(\tilde s(x',y'))\tilde h(x',y') d\alpha' \\&= \int g(\tilde s(x',y'))\tilde h(x',y') d\beta' = \int g(x) d\beta
\end{align*}
and similarly for $(C2)$ and $(C3)$. In particular we need that
$(g \circ \tilde s)\tilde h$, $(g \circ \tilde t)\tilde h$ and $(g \circ \tilde s)(\tilde t - \tilde s)\tilde h$ are competitorblind functions. From Proposition
\ref{prop:invcomprep} we then get that $\tilde s$ only depends on $x$, and $\tilde t$ and $\tilde h$ only depend on $y$ with
the representations given by Lemma \ref{lem:tdependence}. Now, $s$ and $t$ are the inverse functions of $\tilde s$ and $\tilde t$ and $h(x,y) = 1/\tilde h(T(x,y))$.

In the case where we have that $\tilde h$ is constant and $\tilde t$ is affine, 
then we also have that $h$ is constant and $t$ is affine.

In the case where we have that $\tilde h$ is linear and $\tilde t$ is affine in $1/y'$ we now get that
$t$ is of the form $t(y) = a/(y - b)$ and  $h(y) = c(y-b)$.
\end{proof}

\subsection{Pushing Monotonicity Along Transformations}

We are interested in competitor preserving transformations because they allow us to identify
monotone sets for new cost functions derived from such transformations in the following way:

\begin{proposition}\label{prop:monpush}
Let $(T,h)$ be a  monotonicity preserving transformation  and let $c$ be a cost function. 
Define a new cost function $c'(x',y') := \frac{c(T^{-1}(x',y'))}{h(T^{-1}(x',y'))}$. If $\Xi$ is $c$-monotone, then $\Xi' := T(\Xi)$ is $c'$-monotone.
\end{proposition}

\begin{proof}
Let $\alpha'$ be a finite measure concentrated on $\Xi'$ and $\beta'$ be a competitor of $\alpha'$.
Then $\alpha := \tau^{-1}(\alpha')$ and $\beta := \tau^{-1}(\beta')$ are competitors such that
$\alpha$ is concentrated on $\Xi$. Therefore we have that $\alpha'(c') = \alpha(c) \geq \beta(c) = \beta'(c')$.
As $\alpha'$ and $\beta'$ were arbitrary, we have that $\Xi'$ is a $c'$-monotone set.
\end{proof}

\begin{example}
Let $c$ be a Spence--Mirrlees cost function (i.e. $c_{xyy} < 0$). We know that a set $\Xi$ is $c$-monotone if and only if
it is left monotone. We consider the simple transformation $T(x,y) = (-x,y)$ and $h(y) = 1$. Then $(T,h)$ is monotonicity preserving by Theorem \ref{thm:montrafo} with respect to the transformed cost function $c'(x',y') = c(-x',y')$. Observe that we have $c'_{x'y'y'} > 0$.
The set $\Xi' := T(\Xi) = \{(x',y'): (-x',y') \in \Xi\}$ is $c'$-monotone by   Proposition \ref{prop:monpush}. Indeed we can easily
check that $\Xi'$ is right monotone: Suppose we have $(x_1,y_1),(x_1,y_2),(x_2,y') \in \Xi'$ with $x_2 < x_1$ and $y_1 < y' < y_2$,
then we have $(-x_1,y_1),(-x_1,y_2),(-x_2,y') \in \Xi$ with $-x_1 < -x_2$ which contradicts the left-monotonicity of $\Xi$.
\end{example}

\begin{remark}
The example above holds true replacing $T$ by $\tilde T(x,y)=(s(x),y)$ for some differentiable function $s$ such that $s'(x) < 0$ on the support of $\mu$.
\end{remark}

\begin{example}
Let $c(x,y) = -|x-y|$ and again set $T(x,y) = (-x,y)$ and $h(y) = 1$. Then $c'(x,y) = -|x+y|$. If $\mu$ is continuous and
$\Xi$ is a monotone set for $c$, then it is concentrated on two graphs of increasing functions as established above. Then
$T(\Xi)$ is concentrated on two graphs of decreasing functions which corresponds to the shape of monotone sets for this cost
function established in \cite{St14}.
\end{example}

\begin{example}
The transformation described in \cite{CaLaMa16} is given by $T(x,y) = (1/x,1/y)$ with $h(y) = y$. For the cost
function $c(x,y) = -|x-y|$ this yields the transformed cost $c'(x,y) = -|1/x - 1/y|/(1/y) = -|y/x - 1|$ (for $x,y > 0$).
Furthermore the graph of an increasing function is transformed into the graph of an increasing function under this map $T$,
which shows that the optimizer of $-|y/x - 1|$ is also concentrated on two graphs of increasing functions. Similarly
one can deal with $c(x,y) = |x-y|$ and cost functions of Spence--Mirrlees type.
\end{example}

\subsection{Transformations of Martingales}

The given transformations only transform monotone  sets into monotone sets for a modified
cost function. They do not necessarily transform optimal martingale transports into optimal martingale transports.
It is possible to have a monotone set that is not even capable of supporting a martingale, e.g.\ $\Xi=\{(0,1)\}$. Nevertheless if
a martingale is concentrated on a $c$-monotone set (for sufficiently nice $c$), it is optimal. For our transformation to preserve martingales we would have to ask for a more stringent
condition. Namely we would need
$\int g(s(x))(t(y) - s(x))h(y) d\pi(x,y) = 0$ to hold for arbitrary martingales and bounded continuous
functions $g$. This can only hold if $(t(y) - s(x))h(y) = r(x)(y-x)$ which by a simple analysis shows that this
is only possible for concatenations of transformations of the form $T(x,y) = (ax+b,ay+b), h(y)=c$ or $T(x,y) = (1/x,1/y)$ and $h(y) = cy$ which proves necessity in Theorem \ref{thm:change of numeraire}. It is trivial to see that transformations of the first form preserve martingales, while it was shown in \cite{CaLaMa16} that transformations of the second form also
preserve martingales.

\subsection{The Skorokhod Picture of Competitor-Preserving Transformations}

In this section we show that there are clear geometric interpretations of monotonicity preserving transformations which can be well understood using the relation of MOT and SEP.

Let us consider a cost function $c$ satisfying the generalized Spence--Mirrlees condition $c_{xyy}<0$. By Proposition \ref{prop:SpMi} we know that the corresponding SEP solution is the hitting time of a right barrier (see Remark \ref{rem:leftmonotone}) in the phase space $(B_t-B_0,B_t)$. The map $T(x,y)=(-x,-y)$ the point reflection at $(0,0)$ transforms the right barrier into a left barrier, similarly for the transformation $T(x,y)=(1/x,1/y).$ This barrier corresponds to solutions for cost functions $c'$ satisfying $c'_{xyy}>0$ or equivalently to the maximization problem with respect to the cost function $c$ with $c_{xyy}<0$ (see Remark \ref{rem:leftbarrier}).

Considering the transformations $T(x,y)=(-x,y)$, the reflection at the $y$-axis, again a right barrier will be mapped into a left barrier in the $(B_t-B_0,B_t)$ phase space. The transformed cost function satisfies $c'(x,y)=c(-x,y)$. Specifying $c(x,y)=|y-x|$ we get $c'(x,y)=|y+x|$. Moreover, writing the corresponding SEP solution in the $(B_t+B_0,B_t)$ phase space we end up with a right barrier. To be more specific, if  the original optimal
stopping time is $\tau = \inf \{ t\geq 0: B_t - B_0 \geq \psi(B_t)\}$ then the modified stopping
rule is given by $\tau' = \inf \{ t \geq 0: B_t + B_0 \geq \psi(B_t)\}$ which can be rewritten to
$\tau' = \inf \{t \geq 0: B_t - B_0 \leq 2B_t - \psi(B_t)\}$, showing that it is equivalent to a left barrier in the $(B_t-B_0,B_t)$ phase space  revealing a surprising symmetry between solutions for the cost functions  $\gamma(f,s)=|f(s)-f(0)|$ and $\gamma'(f,s)=|f(s)+f(0)|$. 
\begin{figure}
\begin{tabular}{{m{2.5in}m{0.5cm}m{2in}}}
\begin{tikzpicture}[scale=0.3]
\draw[->] (-8,0) -- (5,0) node[right] {$B_t-B_0$};
\draw[->,thick] (0,-6) -- (0,6) node[above] {$B_t$};

\draw [smooth,variable=\y,domain=-5:5.5] plot ( {4.5+0.02*\y-0.2*\y*\y+0.02*\y*\y*\y+0.02*\y*\y*\y*\y+0.001*\y*\y*\y*\y*\y-0.001*\y*\y*\y*\y*\y*\y},{\y});

\fill [pattern = horizontal lines] (5,-5) -- plot [smooth,variable=\y,domain=-5:5.5] ( {4.6+0.02*\y-0.2*\y*\y+0.02*\y*\y*\y+0.02*\y*\y*\y*\y+0.001*\y*\y*\y*\y*\y-0.001*\y*\y*\y*\y*\y*\y},{\y}) -- (5,5.5) -- cycle;

\draw[<->,color=blue,thick] (-3.5,-4.5) -- (4.7,3.7); %
\draw[<->,color=orange,thick] (-5.7,-4.7) -- (3.65,4.65); %
\draw[<->,color=green,thick] (-7.9,-4.9) -- (2,5); %
\end{tikzpicture}
& $\rightarrow$ & 
\begin{tikzpicture}[scale=0.3]
\draw[->] (-8,0) -- (5,0) node[right] {$B_t+B_0$};
\draw[->] (0,-6) -- (0,6) node[above] {$B_t$};

\draw[-,thick] (-10,-5) -- (7,3.5);
\draw [smooth,variable=\y,domain=-5:5.5] plot ( {4.5+0.02*\y-0.2*\y*\y+0.02*\y*\y*\y+0.02*\y*\y*\y*\y+0.001*\y*\y*\y*\y*\y-0.001*\y*\y*\y*\y*\y*\y},{\y});
\fill [pattern = horizontal lines] (5,-5) -- plot [smooth,variable=\y,domain=-5:5.5] ( {4.6+0.02*\y-0.2*\y*\y+0.02*\y*\y*\y+0.02*\y*\y*\y*\y+0.001*\y*\y*\y*\y*\y-0.001*\y*\y*\y*\y*\y*\y},{\y}) -- (5,5.5) -- cycle;

\draw[<->,color=blue,thick] (-5.7,-4.7) -- (3.65,4.65); %
\draw[<->,color=orange,thick] (-3.5,-4.5) -- (4.7,3.7); %
\draw[color=green,fill] (6,3) circle (3pt); %
\end{tikzpicture}
\end{tabular}

\begin{center}
\begin{tikzpicture}[scale=0.3]
\draw[->] (-8,0) -- (5,0) node[right] {$B_t-B_0$};
\draw[->,thick] (0,-6) -- (0,6) node[above] {$B_t$};

\draw [smooth,variable=\y,domain=-5.03:5] plot ( {-4.5+1.98*\y+0.2*\y*\y-0.02*\y*\y*\y-0.02*\y*\y*\y*\y-0.001*\y*\y*\y*\y*\y+0.001*\y*\y*\y*\y*\y*\y},{\y});

\fill [pattern = horizontal lines] (-9,-5.03) -- plot [smooth,variable=\y,domain=-5.03:5] ( {-4.6+1.98*\y+0.2*\y*\y-0.02*\y*\y*\y-0.02*\y*\y*\y*\y-0.001*\y*\y*\y*\y*\y+0.001*\y*\y*\y*\y*\y*\y},{\y}) -- (-9,5) -- cycle;

\draw[<->,color=blue,thick] (-3.7,-4.7) -- (5.65,4.65); %
\draw[<->,color=orange,thick] (-5.5,-4.5) -- (2.75,3.75); %
\draw[color=green,fill] (0,3) circle (3pt); %
\end{tikzpicture}
\end{center} \caption{The transformation of monotone sets for the Spence--Mirrlees cost functions.}
\label{fig:mirrleestransform}
\end{figure}
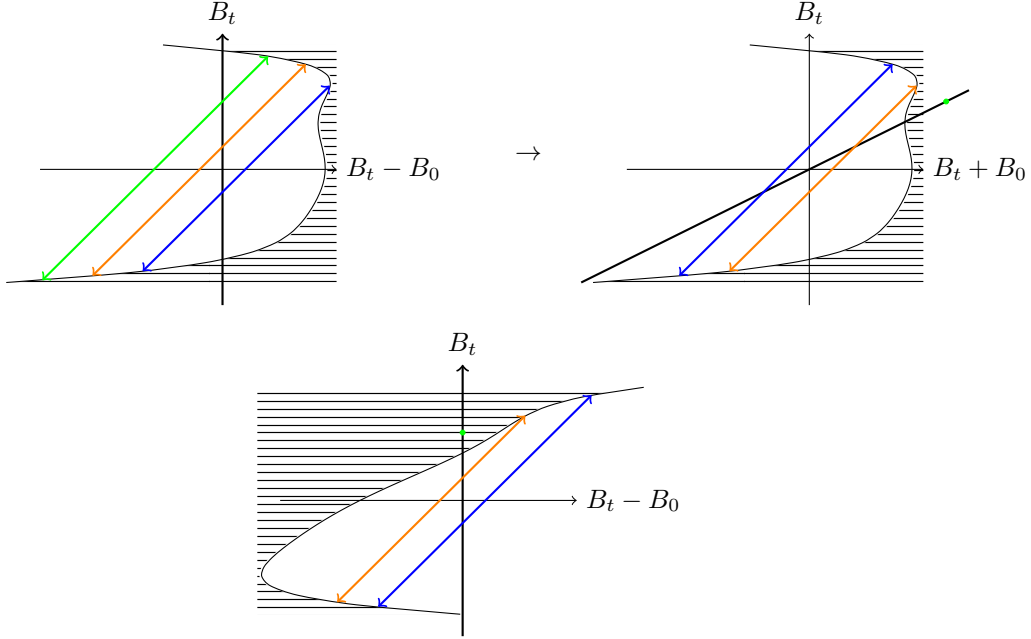
These transformations are depicted in Figure \ref{fig:mirrleestransform}.

Similarly for the cost functional $-|B_0 - B_t|$ we get that the optimal stopping time is given by two barriers
in the $(B_t-B_0,B_t)$ phase space. As before the transformation $T(x,y) = (-x,y)$ transforms these barriers into
two barriers in the $(B_t+B_0,B_t)$ phase space. 
\begin{figure}
\begin{tabular}{m{2in}m{0.5cm}m{2in}}
\begin{tikzpicture}[scale=0.3]
\draw[->] (-6,0) -- (6,0) node[right] {\small $B_t-B_0$};
\draw[->,thick] (0,-6) -- (0,6) node[above] {\small $B_t$};

\draw [smooth,variable=\y,domain=-5:5.5] plot ( {3.5+0.01*\y-0.05*\y*\y+0.0025*\y*\y*\y+0.00125*\y*\y*\y*\y+0.00003125*\y*\y*\y*\y*\y-0.000015625*\y*\y*\y*\y*\y*\y},{\y});
\draw [smooth,variable=\y,domain=-5:5.5] plot ( {-2.8+0.01*\y-0.05*\y*\y+0.0025*\y*\y*\y+0.00125*\y*\y*\y*\y+0.00003125*\y*\y*\y*\y*\y-0.000015625*\y*\y*\y*\y*\y*\y},{\y});

\fill [pattern = horizontal lines] (5,-5) -- plot [smooth,variable=\y,domain=-5:5.5] ( {3.6+0.01*\y-0.05*\y*\y+0.0025*\y*\y*\y+0.00125*\y*\y*\y*\y+0.00003125*\y*\y*\y*\y*\y-0.000015625*\y*\y*\y*\y*\y*\y},{\y}) -- (5,5.5) -- cycle;

\fill [pattern = horizontal lines] (-5,-5) -- plot [smooth,variable=\y,domain=-5:5.5] ( {-2.9+0.01*\y-0.05*\y*\y+0.0025*\y*\y*\y+0.00125*\y*\y*\y*\y+0.00003125*\y*\y*\y*\y*\y-0.000015625*\y*\y*\y*\y*\y*\y},{\y}) -- (-5,5.5) -- cycle;

\draw[<->,color=blue,thick] (-3.9,-4.9) -- (3.3,2.3); %
\draw[<->,color=orange,thick] (-3.1,-2.1) -- (3.2,4.2); %
\draw[<-,color=green,thick] (-2.8,0.2) -- (2.5,5.5); %
\end{tikzpicture}
& $\rightarrow$ &
\begin{tikzpicture}[scale=0.3]
\draw[->] (-8,0) -- (5,0) node[right] {\small $B_t+B_0$};
\draw[->] (0,-6) -- (0,6) node[above] {\small $B_t$};

\draw[-,thick] (-6,-3) -- (7,3.5);
\draw [smooth,variable=\y,domain=-5:5.5] plot ( {3.5+0.01*\y-0.05*\y*\y+0.0025*\y*\y*\y+0.00125*\y*\y*\y*\y+0.00003125*\y*\y*\y*\y*\y-0.000015625*\y*\y*\y*\y*\y*\y},{\y});
\draw [smooth,variable=\y,domain=-5:5.5] plot ( {-2.8+0.01*\y-0.05*\y*\y+0.0025*\y*\y*\y+0.00125*\y*\y*\y*\y+0.00003125*\y*\y*\y*\y*\y-0.000015625*\y*\y*\y*\y*\y*\y},{\y});

\fill [pattern = horizontal lines] (5,-5) -- plot [smooth,variable=\y,domain=-5:5.5] ( {3.6+0.01*\y-0.05*\y*\y+0.0025*\y*\y*\y+0.00125*\y*\y*\y*\y+0.00003125*\y*\y*\y*\y*\y-0.000015625*\y*\y*\y*\y*\y*\y},{\y}) -- (5,5.5) -- cycle;

\fill [pattern = horizontal lines] (-5,-5) -- plot [smooth,variable=\y,domain=-5:5.5] ( {-2.9+0.01*\y-0.05*\y*\y+0.0025*\y*\y*\y+0.00125*\y*\y*\y*\y+0.00003125*\y*\y*\y*\y*\y-0.000015625*\y*\y*\y*\y*\y*\y},{\y}) -- (-5,5.5) -- cycle;

\draw[<->,color=blue,thick] (-3.1,-2.1) -- (3.2,4.2); %
\draw[<->,color=orange,thick] (-3.9,-4.9) -- (3.3,2.3); %
\draw[color=green,fill] (6,3) circle (3pt); %
\end{tikzpicture}
\end{tabular}

\begin{tikzpicture}[scale=0.3]
\draw[->] (-8,0) -- (5,0) node[right] {\small $B_0$};
\draw[->] (0,-6) -- (0,6) node[above] {\small $B_t$};
\draw[-,thick] (-5,-5) -- (5,5);

\draw [smooth,variable=\y,domain=-1:5.5] plot ( {3.5-0.99*\y-0.05*\y*\y+0.0025*\y*\y*\y+0.00125*\y*\y*\y*\y+0.00003125*\y*\y*\y*\y*\y-0.000015625*\y*\y*\y*\y*\y*\y},{\y});
\draw [smooth,variable=\y,domain=-5:1] plot ( {-2.8-0.99*\y-0.05*\y*\y+0.0025*\y*\y*\y+0.00125*\y*\y*\y*\y+0.00003125*\y*\y*\y*\y*\y-0.000015625*\y*\y*\y*\y*\y*\y},{\y});

\fill [pattern = horizontal lines] (4.53,5.5) -- plot [smooth,variable=\y,domain=-1:5.5] ( {3.6-0.99*\y-0.05*\y*\y+0.0025*\y*\y*\y+0.00125*\y*\y*\y*\y+0.00003125*\y*\y*\y*\y*\y-0.000015625*\y*\y*\y*\y*\y*\y},{\y}) -- cycle;

\fill [pattern = horizontal lines] (-3.94,-5) -- plot [smooth,variable=\y,domain=-5:1] ( {-2.9-0.99*\y-0.05*\y*\y+0.0025*\y*\y*\y+0.00125*\y*\y*\y*\y+0.00003125*\y*\y*\y*\y*\y-0.000015625*\y*\y*\y*\y*\y*\y},{\y}) -- cycle;

\draw[<->,color=blue,thick] (-1,-2) -- (-1,4.2); %
\draw[<->,color=orange,thick] (1,-4.9) -- (1,2.3); %
\draw[color=green,fill] (3,3) circle (3pt); %
\end{tikzpicture}

 \caption{The transformation of monotone sets for the cost function $c(x,y) = -|x-y|$.}
\label{fig:neubergertransform}
\end{figure}
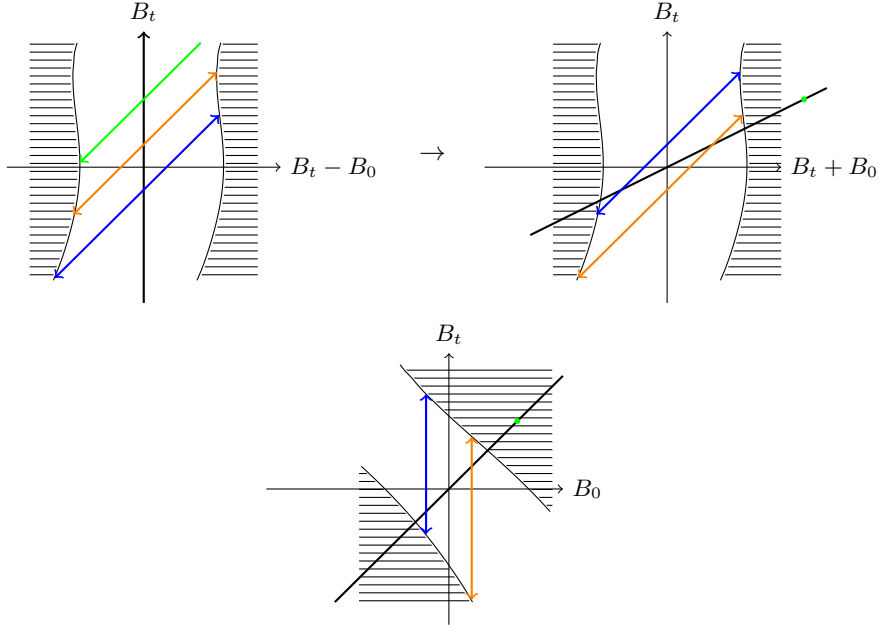
This is depicted in Figure \ref{fig:neubergertransform}.

We  emphasize that this form of transformations does not lead to the optimal stopping time for
the same marginals $\mu,\nu$ that the original barrier was constructed for. If we start an arbitrary
distribution $\mu$ and stop it at this barrier, it will be the optimal stopping time between its marginals.

\section{Towards a 'full' Monotonicity Principle for SEP}\label{sec:conj}

Some transformations discussed in Section \ref{sec:trafo} have clear analogues for \eqref{eq:OptSEP}, e.g.\ $(x,y)\mapsto (1/x,1/y)$ corresponds to $(f,s)\mapsto (1/f,s)$ or $(x,y)\mapsto (ax+b,ay+b)$ corresponds $(f,s)\mapsto (af+b,s)$, whereas the analogues of other transformations might be less obvious. Nevertheless, the 'good' transformations reveal symmetries between solutions to \eqref{eq:OptSEP} for different cost functions $\gamma$. To establish an analogue to Proposition \ref{prop:monpush} for \eqref{eq:OptSEP} it seems to be necessary to find and prove a full monotonicity principle. Comparing with MOT we now give a conjecture on this full monotonicity principle. To this end, we need the following definition.

\begin{definition}
 Let $\alpha$ and $\beta$ be two finite measures on $S$. Put $h^\mathsf{initial}(f,s)=f(0)$ and $h^\mathsf{final}(f,s)=f(s)$. We say that $\beta$ is a competitor of $\alpha$ iff 
\begin{align}
 (h^\mathsf{initial})_*\alpha &= (h^\mathsf{initial})_*\beta  \label{eq:SEPfirstmarginal} \\
 (h^\mathsf{final})_*\alpha &= (h^\mathsf{final})_*\beta\ ,\label{eq:SEPsecmarginal}
\end{align}
 and for all martingales $\phi$ such that there is a continuous function $H:S\to \R$ such that $\phi(\omega,t)=H(\omega_{\llcorner [0,t]},t)$
\begin{align}
 \int \phi\ d\alpha:= \int H\ d\alpha = \int H\ d\beta =: \int \phi\ d\beta\ . \label{eq:SEPdisintegration}
\end{align}

\end{definition}
Conditions \eqref{eq:SEPfirstmarginal} and \eqref{eq:SEPsecmarginal} are the obvious analogues of the equal marginal constraint in Definition \ref{def:monotone}. The condition \eqref{eq:SEPdisintegration} should be compared to \eqref{eq:compdef}. Since a full monotonicity principle for \eqref{eq:OptSEP} needs to be able to make pathwise comparisons we only require \eqref{eq:SEPdisintegration} for functions which are well defined on a pointwise level. Examples for competitors can be read off from Definition \ref{def:SG}, i.e.\ given a stopping time $\sigma$ with positive and finite expectation we set
\begin{align*}
\alpha&=\delta_{(g,t)} + \int \delta_{(f,s)\oplus ((B_u)_{u\leq\sigma}(\omega),\sigma(\omega))}\P(d\omega)\\
 \beta &=\delta_{(f,s)} + \int \delta_{(g,t)\oplus ((B_u)_{u\leq\sigma}(\omega),\sigma(\omega))}\P(d\omega)\ .
\end{align*}

The following should be compared to  Definition \ref{def:SG}(ii).
\begin{definition}
 We say that a set $\Gamma\subset S$ is \emph{strongly $\gamma$-monotone} if for any finite measure $\alpha$ concentrated on $\Gamma$ such that $(h^\mathsf{initial})_*\alpha$ is concentrated on finitely many points  and any competitor $\beta$ of $\alpha$ it holds that 
$$\int \gamma\ d\alpha \leq \int \gamma\ d\beta.$$
\end{definition}

\begin{conjecture*}
Let $\gamma: S \to \R$ be Borel $\mu, \nu$ be two probabilities on $\R$ increasing in convex order. There exists a strongly $\gamma$-monotone set $\Gamma \subseteq S$ such that a solution $\tau$ of SEP is a solution to \eqref{eq:OptSEP} if and only if $\P[(B_\tau,\tau)\in \Gamma] = 1$.
\end{conjecture*}

We only note here that the conjecture holds for functionals $\gamma(f,s)=c(f(0),f(s))$ due to the intimate connection of MOT and SEP and the respective result for MOT, Theorem \ref{thm:fullcm}.

\appendix

\section{Proof of Lemma \ref{lem:tdependence}}

\begin{proof}
We first write $\tilde t(x,y)^2(c(x) + d(x)y) = \phi(x) + \psi(y) + k(x)y$ and multiply both sides by $c(x) + d(x)y$ (which was assumed
to be positive everywhere) to obtain
\[a(x)^2 + 2a(x)b(x)y + b(x)^2y^2 = (c(x) + d(x)y)(\phi(x) + \psi(y) + k(x)y).\]
As  the left side is a polynomial in $y$, the right side must also be a polynomial in $y$ which can only hold
if we can write $\psi(y) = p_{-1} y^{-1} + p_2y^2$ (we can assume that no constant and linear term exists by modifying
$\phi$ and $k$ accordingly). 

Next observe that we cannot have both $p_{-1} \neq 0$ and $p_2 \neq 0$. In this case  we must have $c(x) = d(x) = 0$
for all $x$ so that we do not have a term depending on $y^3$ nor a term that depends on $y^{-1}$ on the right side which
is necessary for this equality to hold. However, this  contradicts the assumption that $c(x) + d(x)y > 0$.

Furthermore, it is also impossible to have $p_{-1} = p_2 = 0$. In this case the above equation simplifies to
\begin{align*}
a(x)^2 &+ 2a(x)b(x)y + b(x)^2y^2 \\
&= c(x)\phi(x) + [c(x)k(x) + d(x)\phi(x)]y + d(x)k(x)y^2
\end{align*}
which by comparison of coefficients gives the conditions
$a(x)^2 = c(x)\phi(x)$, $b(x)^2 = d(x)k(x)$ and $2a(x)b(x) = c(x)k(x) + d(x)\phi(x)$.
Suppose for fixed $x$ that $c(x) = 0$, then we also have $a(x) = 0$ and therefore $\tilde t(x,y) = b(x)/d(x)$ does not
depend on $y$ which contradicts the assumption that $t$ is non-constant in $y$ for this $x$. Similarly
$d(x) = 0$ implies $b(x) = 0$ and therefore $\tilde t(x,y) = a(x)/c(x)$ which again would be constant in $y$.
Therefore we can assume that $c(x)d(x) \neq 0$ and multiply the last condition by this term to obtain
$2a(x)b(x)c(x)d(x) = c(x)^2k(x)d(x) + d(x)^2\phi(x)c(x)$. Now substitute the first two conditions to obtain
$2a(x)b(x)c(x)d(x) = c(x)^2b(x)^2 + d(x)^2a(x)^2$ which implies that $b(x)c(x) = a(x)d(x)$ holds. 
A M\"obius-transform with coefficients whose determinant is $0$ is constant (in $y$) which 
would mean that $t$ is constant in $y$.

Now we consider the cases where exactly one of the $p_i$ is $0$:

If $p_2 = 0$ but $p_{-1} \neq 0$, we must have $c(x) = 0$ for all $x$ because otherwise the right side
of the above equation has a term that depends on $y^{-1}$ which does not occur on the left side. Then the equation
simplifies to
\[a(x)^2 + 2a(x)b(x)y + b(x)^2y^2 = d(x)p_{-1} + \phi(x)d(x)y + d(x)k(x)y^2\]
which by comparing the coefficient of the linear term gives $a(x)^2 = d(x)p_{-1}$.
Substituting $d(x)$ in the definition of $\tilde t(x,y)$ accordingly gives
$\tilde t(x,y) = \frac{p_{-1}}{a(x)}y^{-1} + \frac{p_{-1}b(x)}{a(x)^2}$. Now considering 
$\tilde t(x,y)^3d(x)y = \tilde t(x,y)^3\frac{a(x)^2}{p_{-1}}y$ we have that the coefficient of $y^{-2}$ is given by
$p_{-1}^2/a(x)$ and needs to be constant in $x$ in order for this function to be competitorblind which shows that $a(x)$ is
constant in $x$. Now writing $a(x) = a$ we also have that the coefficient of $y^{-1}$ is given by $3p_{-1}^2b(x)/a^2$ which
also needs to be constant in $x$ and thus $b(x)$ is constant in $x$. In conclusion we have that $\tilde t(x,y)$ is constant in $x$ and affine in $y^{-1}$.

If $p_{-1} = 0$ but $p_2 \neq 0$, we must have $d(x) = 0$ for all $x$ because otherwise the right side
of the above equation has a term that depends on $y^3$ which does not occur on the left side. Then the equation
simplifies to
\[a(x)^2 + 2a(x)b(x)y + b(x)^2y^2 = c(x)\phi(x) + c(x)k(x)y + c(x)p_2y^2\]
for which we can compare the coefficient of the square term to obtain the condition $b(x)^2 = c(x)p_2$. Substituting for
$c(x)$ in the definition of $\tilde t(x,y)$ then gives $\tilde t(x,y) = \frac{p_2a(x)}{b(x)^2}+\frac{p_2}{b(x)}y$. We have that
$\tilde t(x,y)^3c(x) = \tilde t(x,y)^3b(x)^2/p_2$ is competitorblind and thus its coefficient of $y^3$ must be constant in $x$. This 
coefficient is given by $\frac{p_2^2}{b(x)}$ and therefore $b(x) = b$ is constant in $x$. The coefficient of $y^2$ of this
function is given by $3\frac{p_2^2 a(x)}{b^2}$ and must also be constant in $x$ which implies $a(x) = a$ is constant in $x$.
We can conclude as in the first case that $\tilde t(x,y)$ is constant in $x$ and affine in $y$.
\end{proof}

\bibliography{joint_biblio}{}
\bibliographystyle{plain}

\end{document}